\documentclass[a4paper]{scrartcl}

\usepackage[utf8]{luainputenc}
\usepackage[USenglish]{babel}
\usepackage{csquotes}

\usepackage[a4paper, top=2.5cm, bottom=2.5cm, left=2.5cm, right=2.5cm]{geometry}
\usepackage[plainpages=false,pdfpagelabels,hidelinks,unicode]{hyperref}

\usepackage[%
  backend=biber,
  style=numeric-comp,
  giveninits=true, uniquename=init, %abbreviate first names
  natbib=true,
  url=true,
  doi=true,
  isbn=false,
  backref=false,
  maxnames=99,
  ]{biblatex}
\usepackage{amsmath}
\numberwithin{equation}{section}
\allowdisplaybreaks
\usepackage{amssymb}
\usepackage{commath}
\usepackage{mathtools}
\usepackage{bbm}
\usepackage{nicefrac}

\usepackage{siunitx}

\usepackage{amsthm}
\theoremstyle{plain}
  \newtheorem{theorem}{Theorem}
  \newtheorem{lemma}[theorem]{Lemma}

\theoremstyle{definition}
  
  \newtheorem{remark}[theorem]{Remark}
  
\numberwithin{theorem}{section}

\usepackage{color}
\usepackage{graphicx}
\usepackage[small]{caption}
\usepackage{subcaption}

\begingroup\expandafter\expandafter\expandafter\endgroup
\expandafter\ifx\csname pdfsuppresswarningpagegroup\endcsname\relax
\else
\fi

\usepackage{booktabs}
\usepackage{rotating}
\usepackage{multirow}

\usepackage{ifluatex}
\ifluatex
  \usepackage[no-math]{fontspec}
\else
  \usepackage[T1]{fontenc}
\fi
\usepackage{newpxtext,newpxmath}

\newcommand{\N}{\mathbb{N}}

\renewcommand{\O}{\mathcal{O}}
\newcommand{\dt}{{\Delta t}}

\newcommand{\scp}[2]{\left\langle{#1,\, #2}\right\rangle}
\newcommand{\I}{\operatorname{I}}
\newcommand{\lot}{\ell_{\mathrm{ot}}}

\usepackage{forest}
\forestset{
  */.style={
    delay+={append={[]},}
  },
  rooted tree/.style={
    for tree={
      grow'=90,
      parent anchor=center,
      child anchor=center,
      s sep=2.5pt,
      if level=0{
        baseline
      }{},
      delay={
        if content={*}{
          content=,
          append={[]}
        }{}
      }
    },
    before typesetting nodes={
      for tree={
        circle,
        fill,
        minimum width=3pt,
        inner sep=0pt,
        child anchor=center,
      },
    },
    before computing xy={
      for tree={
        l=5pt,
      }
    }
  }
}
\DeclareDocumentCommand\rootedtree{o}{\Forest{rooted tree [#1]}}

\DeclarePairedDelimiter\floor{\lfloor}{\rfloor}

\newenvironment{keywords}{\par\textbf{Key words.}}{\par}
\newenvironment{AMS}{\par\textbf{AMS subject classification.}}{\par}

\title{Relaxation Runge--Kutta Methods\texorpdfstring{\\}{ }for Hamiltonian Problems}
\author{Hendrik Ranocha, David I. Ketcheson}
\date{June 4, 2020}

\makeatletter
\hypersetup{pdfauthor={\@author}}
\hypersetup{pdftitle={\@title}}
\makeatother

\begin{document}

\maketitle

\begin{abstract}
The recently-introduced relaxation approach for Runge--Kutta methods can
be used to enforce conservation of energy in the integration of Hamiltonian
systems. We study the behavior of implicit and explicit relaxation
Runge--Kutta methods in this context. We find that, in addition to their
useful conservation property, the relaxation methods yield other improvements.
Experiments show that their solutions bear stronger qualitative similarity
to the true solution and that the error grows more slowly in time. We also
prove that these methods are superconvergent for a certain class of Hamiltonian
systems.

\end{abstract}


\begin{AMS}
  65L06,  % NA, ODEs: Multistep, Runge-Kutta and extrapolation methods
  65L20,  % NA, ODEs: Stability and convergence of numerical methods
  65M12,  % NA, PDEs, IVPs, IBVPs: Stability and convergence of numerical methods
  65M70,  % NA, PDEs, IVPs, IBVPs: Spectral, collocation and related methods
  65P10,  % NA, Numerics for dynamical systems: Hamiltonian systems including symplectic integrators
  37M99   % Dynamical systems and ergodic theory, Approximation methods and numerical treatment of dynamical systems: None of the above, but in this section
\end{AMS}

\section{Introduction}
\label{sec:introduction}
Many important differential equations possess one or more quantities,
known as invariants, that remain constant in time. Common examples include
total mass, momentum, or energy. In the numerical integration
of initial value problems, it is often important
that this invariance be preserved, not just
to the accuracy of truncation error, but up to machine precision.
The term \emph{geometric numerical integration} is widely used to refer
to numerical methods that preserve invariants, and such methods are
the subject of much study; see \cite{hairer2006geometric} and references therein.

Linear invariant quantities (such as total mass in many models) are preserved
by all general linear methods, including all Runge-Kutta and linear multistep methods.
Nonlinear invariants are not generally preserved by such methods, and in
particular they are not preserved by explicit or diagonally implicit
Runge-Kutta methods, which are otherwise the methods of choice for many
classes of initial value problems.

In the present work, we focus on a simple modification that can be used
to make any Runge-Kutta method preserve a desired nonlinear invariant,
while retaining its order of accuracy and other useful properties.
The resulting methods, known as relaxation Runge-Kutta (RRK) methods,
have been developed and analyzed in
\cite{ketcheson2019relaxation,ranocha2020relaxation}, where
the focus was on preserving dissipative properties with explicit methods. Herein we study
in detail this approach applied to problems with a nonlinear conserved quantity;
we also consider implicit RRK methods. We focus specifically on problems
that can be written as a Hamiltonian system:
\begin{equation}
\label{eq:Hamiltonian}
  \od{}{t} q(t) = \partial_p H\bigl( q(t), p(t) \bigr),
  \quad
  \od{}{t} p(t) = - \partial_q H\bigl( q(t), p(t) \bigr).
\end{equation}
Here $H(q,p)$ is the Hamiltonian and often represents energy.
We refer to methods that preserve $H$ (to machine
precision) as \emph{energy-preserving}.

\subsection{Energy-Preserving Methods}
Symplectic Runge-Kutta methods preserve arbitrary quadratic invariants \cite{cooper1987stability},
and possess other properties desirable in geometric integration. However,
no Runge-Kutta method preserves arbitrary nonlinear invariants \cite{iserles2000preserving}.
Nevertheless, there exist various Runge-Kutta-like methods that are
energy-preserving.

Projection methods (see \cite[Section IV.4]{hairer2006geometric}) achieve
energy preservation by projecting the solution orthogonally onto the energy-conservative manifold
at the end of each step. This projection destroys linear covariance and has been
shown to give poor long-term results for some problems. Furthermore, the construction
of the method is specific to a given problem. Discrete gradients
\cite{gonzalez1996time,mclachlan1999geometric} provide a more systematic approach,
based on introducing a carefully-designed discrete approximation of the gradient
of the Hamiltonian, yielding a discrete analog of the chain rule that ensures energy
conservation. More recently, this has led to the averaged vector field (AVF)
method \cite{quispel2008new} and subsequent approaches referred to as
continuous-stage Runge-Kutta (CSRK) methods \cite{hairer2010energy}.
Methods based on the AVF or CSRK approach have a computational cost similar to that
of fully-implicit Runge-Kutta methods, as they require the solution of
a system of algebraic equations of the size of the number of ODEs multiplied by the
number of stages.
In contrast, projection and RRK methods based on explicit Runge-Kutta methods
require only the solution of a single scalar algebraic equation at each time step;
if the Hamiltonian is generated by an inner product, then this equation has
a known explicit solution.
RRK methods based on diagonally implicit Runge-Kutta methods also promise to be
more efficient than most of the above-mentioned methods, since the stage equations
can be solved sequentially.

Many energy-preserving methods have been shown to enjoy favorable properties such
as improved overall accuracy and qualitative solution behavior
\cite{calvo1993development,portillo1995lack,cano1997error,calvo2011error}.
It is natural to ask which of the good (or bad) properties of these related
schemes are shared by RRK methods. This is the main focus of the present work.

The remainder of this paper is organized as follows. In Section~\ref{sec:RRK}
we review RRK methods.  In Section~\ref{sec:qual} we study the qualitative
behavior of RRK methods, focusing particularly on the preservation of orbits.
In Section~\ref{sec:accuracy} we study the accuracy of RRK methods for Hamiltonian
systems and prove two advantageous properties in this regard. In Section~\ref{sec:n-body}
we consider the application of RRK methods to $N$-body problems in celestial
mechanics and molecular dynamics. In Section~\ref{sec:conclusion} we present
the outlook for application of RRK methods to Hamiltionian systems and some
planned future work. Proofs of some technical results are found in the appendix.

\section{Relaxation Runge-Kutta Methods}
\label{sec:RRK}

Relaxation Runge-Kutta methods have been developed recently in
\cite{ketcheson2019relaxation,ranocha2020relaxation},
although the idea behind them goes back much further
\cite{dekker1984stability,calvo2006preservation,calvo2010projection}.
Given a Runge-Kutta method
\begin{subequations}
\label{eq:RK-step}
\begin{align}
\label{eq:RK-stages}
  y_i
  &=
  u^n + \dt \sum_{j=1}^{s} a_{ij} \, f(t_n + c_j \dt, y_j),
  \qquad i \in \set{1, \dots, s},
  \\
\label{eq:RK-final}
  u(t_n + \dt) \approx u^{n+1}
  &=
  u^n + \dt \sum_{i=1}^{s} b_{i} \, f(t_n + c_i \dt, y_i),
\end{align}
\end{subequations}
we define
\begin{align}
d^n := \sum_{i=1}^{s} b_{i} f_i,
\end{align}
where we use the shorthand $f_i := f(t_n + c_i \dt, y_i)$.
The relaxation idea is to replace the update formula \eqref{eq:RK-final} with
an update in the same direction but of a (possibly) different length:
\begin{align}
  u(t_n + \gamma^n\dt) \approx u^{n+1}_\gamma &= u^n + \gamma^n \dt d^n.
\end{align}
Here $\gamma^n$ is referred to as the relaxation parameter (by analogy with
iterative algebraic solvers) and is chosen to ensure exact
satisfaction of some qualitative property. In the present context, we
assume there is an invariant $H(u)$ (the Hamiltonian, or energy, of the
system) and choose $\gamma^n$ so that
\begin{align} \label{discrete_invariant}
  H(u^{n+1}_\gamma) = H(u^n).
\end{align}
This is a scalar equation and can usually be efficiently solved (for $\gamma^n$)
with standard rootfinding techniques. Existence of a solution is guaranteed
for small enough $\dt$ and the resulting RRK method (interpreting
$u^{n+1}_\gamma$ as approximation to $u(t_n + \gamma^n \dt)$) is of the
same order as the baseline scheme (see \cite{ranocha2020relaxation}).

The relaxation technique can also be extended to quite general time integration
methods that are at least second-order accurate \cite{ranocha2020general}.
The existence of a (useful) solution for the relaxation parameter $\gamma^n$
is in general only guaranteed if the time step is small enough.
In our experience with many different systems, this bound is usually slightly
bigger than the time step restriction for the baseline scheme to be applied
without drastic stability problems for short-time simulations. While relaxation
does not necessarily increase the maximal time step significantly, it increases
the reliability of the numerical time integrator, both in terms of qualitative
and quantitative behavior, in particular for long-time simulations.

If the underlying Runge-Kutta method is explicit, the RRK method is \emph{almost}
explicit, requiring solution of a single scalar equation per step. If the underlying
Runge-Kutta method is implicit, the additional cost of solving \eqref{discrete_invariant}
is generally quite small compared to the cost of solving the Runge-Kutta stage
equations.

We have implemented the relaxation Runge-Kutta methods used in this article
in Python, using SciPy \cite{virtanen2019scipy} to solve the scalar
non-quadratic equations for the relaxation parameter $\gamma$.
The source code for all numerical examples is available online
\cite{ranocha2020relaxationHamiltonianRepro}.

\section{Qualitative Behavior of Solutions}
\label{sec:qual}

At the most basic level, a numerical method should preserve the asymptotic
behavior of the solution in terms of whether the solution becomes unbounded
or approaches an equilibrium point. Generally, the solution of
a Hamiltonian system does neither of these. Using any energy-preserving method
typically ensures that the numerical solution does not either.

It is desirable to also preserve more specific qualitative properties of the
solution. For instance, is the behavior regular or chaotic?  Does it remain
on or very near a fixed orbit?  Does it remain inside a certain invariant set?
In this section we investigate preservation of such properties by applying
RRK methods to a few representative and well-known examples.

\subsection{Lotka-Volterra Equations}
\label{sec:qualitative-Lotka-Volterra}

While the Lotka-Volterra equations
\begin{equation}
\label{eq:Lotka-Volterra}
  \od{}{t} \begin{pmatrix} u_1(t) \\ u_2(t) \end{pmatrix}
  =
  \begin{pmatrix}
    u_1(t) (1 - u_2(t)) \\
    u_2(t) (u_1(t) - 1)
  \end{pmatrix},
  \quad
  u(0) = \begin{pmatrix} 1 \\ 2 \end{pmatrix},
\end{equation}
are not a canonical Hamiltonian system, they can be transformed into this
classical form by a change of variables and are widely studied in the
context of structure preserving numerical methods
\cite[Section~I.1.1]{hairer2006geometric}. Phase space portraits of
numerical solutions obtained via the classical fourth order method
RK(4,4) of Kutta~\cite{kutta1901beitrag} are shown in
Figure~\ref{fig:Lotka-Volterra}. There, the baseline
scheme yields a numerical approximation spiraling inwards till the final
time $t = \num{500}$. In contrast, the relaxation method preserving the
invariant
\begin{equation}
  H(u) = u_1 - \log(u_1) + u_2 - \log(u_2)
\end{equation}
shows a superior behavior in phase space.

\begin{figure}[t]
\centering
  \begin{subfigure}[b]{0.4\textwidth}
    \centering
    \includegraphics[width=\textwidth]{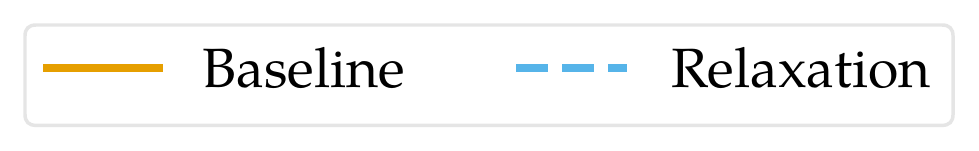}
  \end{subfigure}%
  \\
  \begin{subfigure}[b]{0.40\textwidth}
    \centering
    \includegraphics[width=\textwidth]{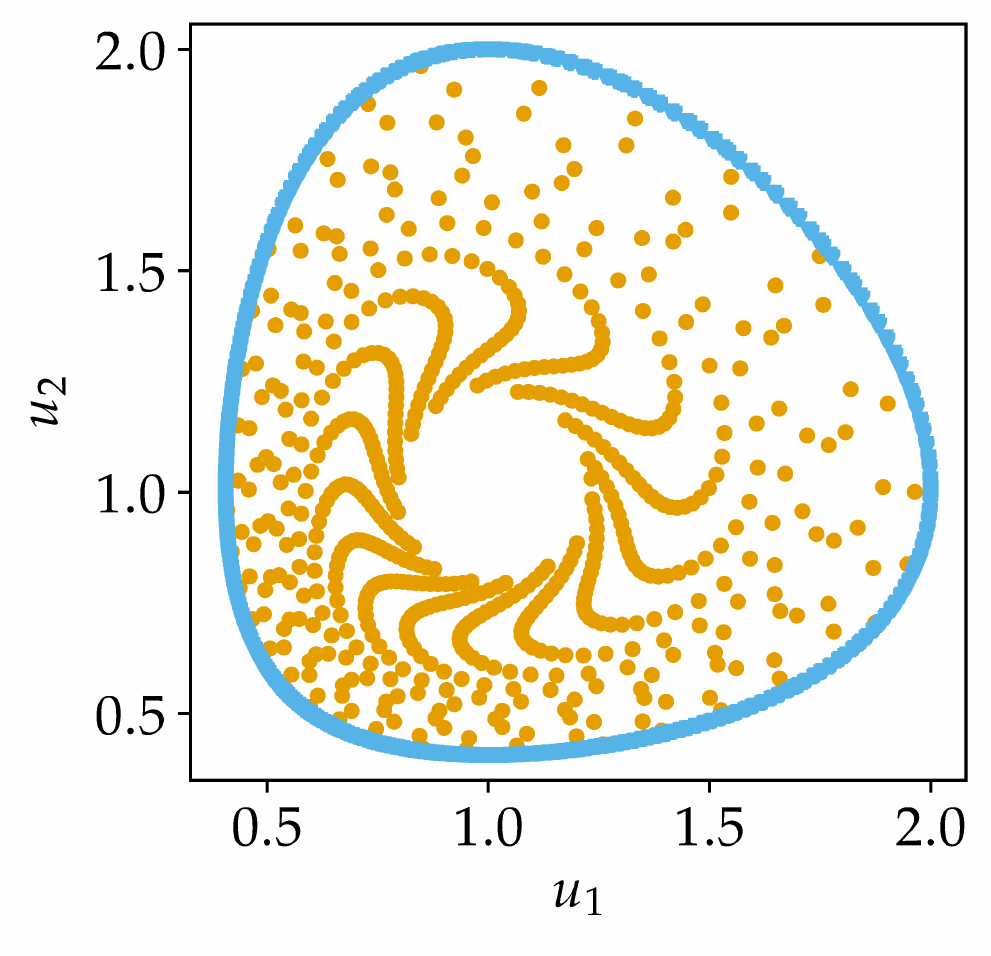}
    \caption{Numerical solutions in phase space.}
  \end{subfigure}%
  ~
  \begin{subfigure}[b]{0.59\textwidth}
    \centering
    \includegraphics[width=\textwidth]{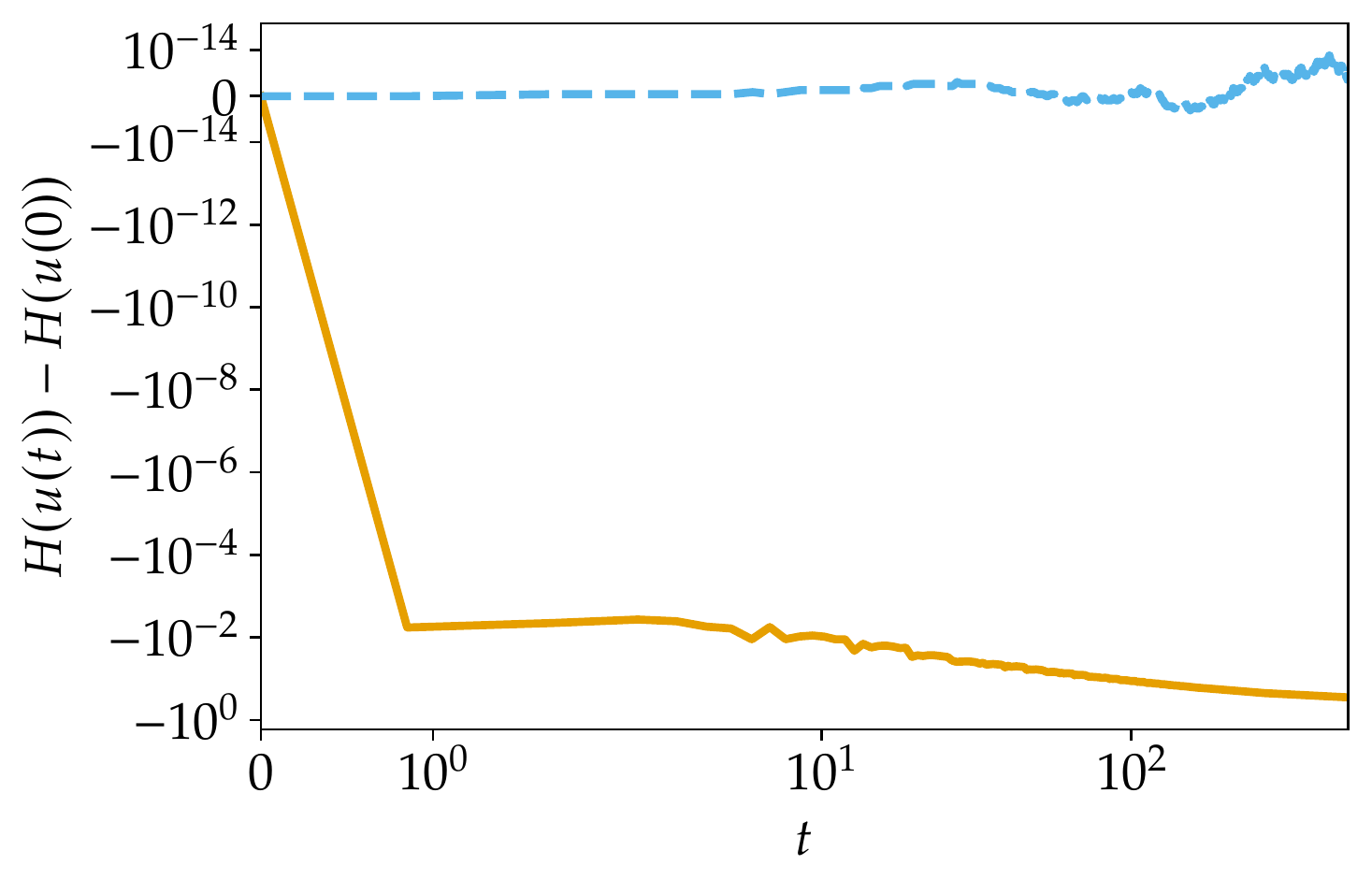}
    \caption{Variation of the Hamiltonian $H$.}
  \end{subfigure}%
  \caption{Simulation of the Lotka-Volterra system \eqref{eq:Lotka-Volterra}
           using the classical RK(4,4) method with and without relaxation
           and a time step $\dt = 0.85$.}
  \label{fig:Lotka-Volterra}
\end{figure}

\subsection{Hénon-Heiles System}
\label{sec:qualitative-Henon-Heiles}
The Hénon-Heiles system is a canonical Hamiltonian system
\eqref{eq:Hamiltonian} with Hamiltonian
\begin{equation}
\label{eq:Henon-Heiles}
  H(q,p) = \frac{p_1^2 + p_2^2}{2}
          + \frac{q_1^2 + q_2^2}{2} + q_1^2 q_2
          - \frac{q_2^3}{3}.
\end{equation}
This non-integrable system is widely studied in the literature; see e.g.
\cite[Section~1.2.6]{sanzserna1994numerical}. Depending on the initial
condition, $(q_2, p_2)$ Poincaré sections in the $q_1 = 0$ plane
can be either chaotic or show structured curves.

For the initial condition
\begin{equation}
\label{eq:Henon-Heiles-quasiperiodic}
  q_1(0) = q_2(0) = p_1(0) = p_2(0) = 0.12,
\end{equation}
the solution is quasiperiodic and two disjoint closed curves
form the Poincaré section in the $q_1 = 0$ plane
\cite[Section~1.2.6]{sanzserna1994numerical}.
This Poincaré section and the time evolution of the Hamiltonian $H$ of
numerical solutions using SSPRK(3,3) of \cite{shu1988efficient}
are shown in Figure~\ref{fig:Henon-Heiles}. Clearly, the baseline scheme
dissipates energy and does not result in a quasiperiodic solution. Its
Poincaré section is totally misleading while the relaxation scheme yields
the correct qualitative behavior.

\begin{figure}[t]
\centering
  \begin{subfigure}[b]{0.4\textwidth}
    \centering
    \includegraphics[width=\textwidth]{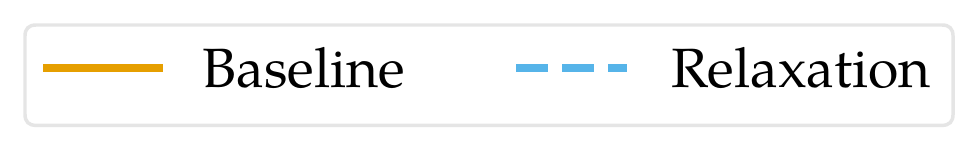}
  \end{subfigure}%
  \\
  \begin{subfigure}[b]{0.40\textwidth}
    \centering
    \includegraphics[width=\textwidth]{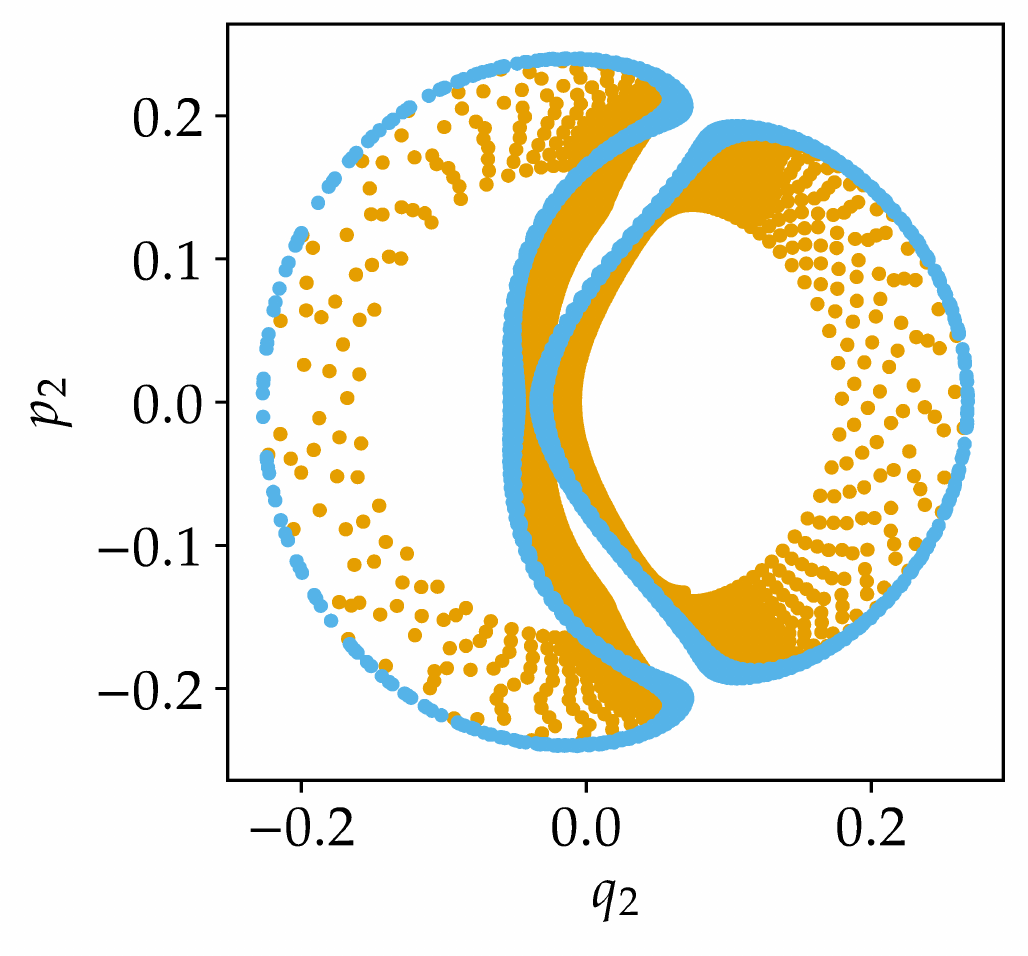}
    \caption{Poincaré sections, $q_1 = 0$.}
  \end{subfigure}%
  ~
  \begin{subfigure}[b]{0.58\textwidth}
    \centering
    \includegraphics[width=\textwidth]{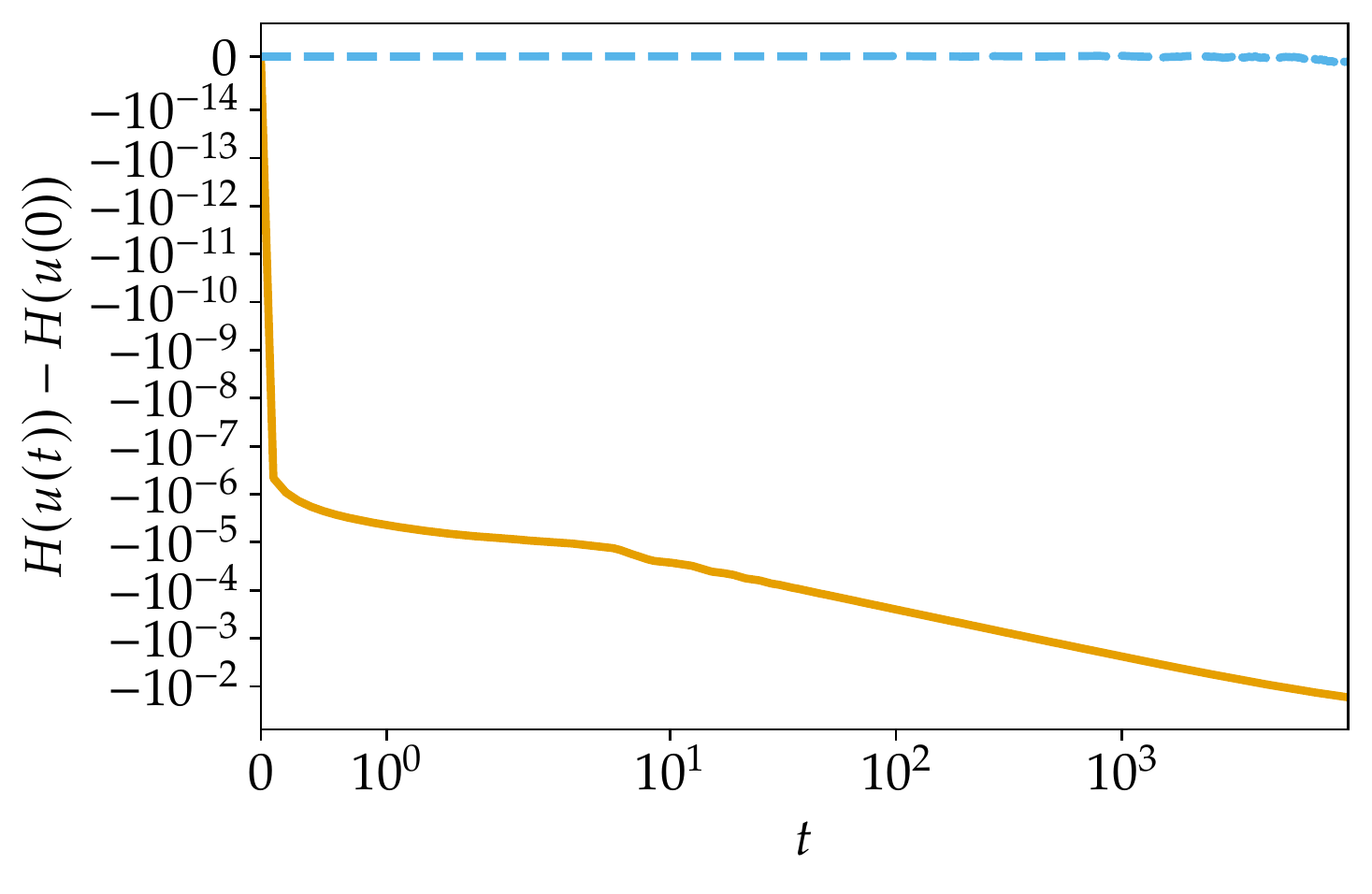}
    \caption{Variation of the Hamiltonian $H$.}
  \end{subfigure}%
  \caption{Simulation of the Hénon-Heiles system \eqref{eq:Henon-Heiles}
           with initial condition \eqref{eq:Henon-Heiles-quasiperiodic}
           using the third order method SSPRK(3,3) of \cite{shu1988efficient}
           with and without relaxation and a time step $\dt = 0.1$.}
  \label{fig:Henon-Heiles}
\end{figure}

Choosing the initial condition
\begin{equation}
\label{eq:Henon-Heiles-chaotic}
  q_1(0) = q_2(0) = p_2(0) = 0.12,
  \quad
  p_1(0) = \sqrt{2} \sqrt{0.15925},
\end{equation}
yields a chaotic solution \cite[Section~1.2.6]{sanzserna1994numerical}.
The Poincaré sections of numerical solutions computed
till the final time $t = \num{30000}$ with and without relaxation are
visualised in Figure~\ref{fig:Henon-Heiles-chaotic}.
Again, the relaxation approach results in an improved qualitative behavior
for SSPRK(3,3).  The Poincaré section
obtained using the relaxation scheme agrees with the reference plot
of \cite[Figure~1.5]{sanzserna1994numerical}. In contrast, the baseline
scheme yields a chaotic Poincaré section that is concentrated in a wrong
area.

\begin{figure}[t]
\centering
  \begin{subfigure}[b]{0.49\textwidth}
    \centering
    \includegraphics[width=\textwidth]{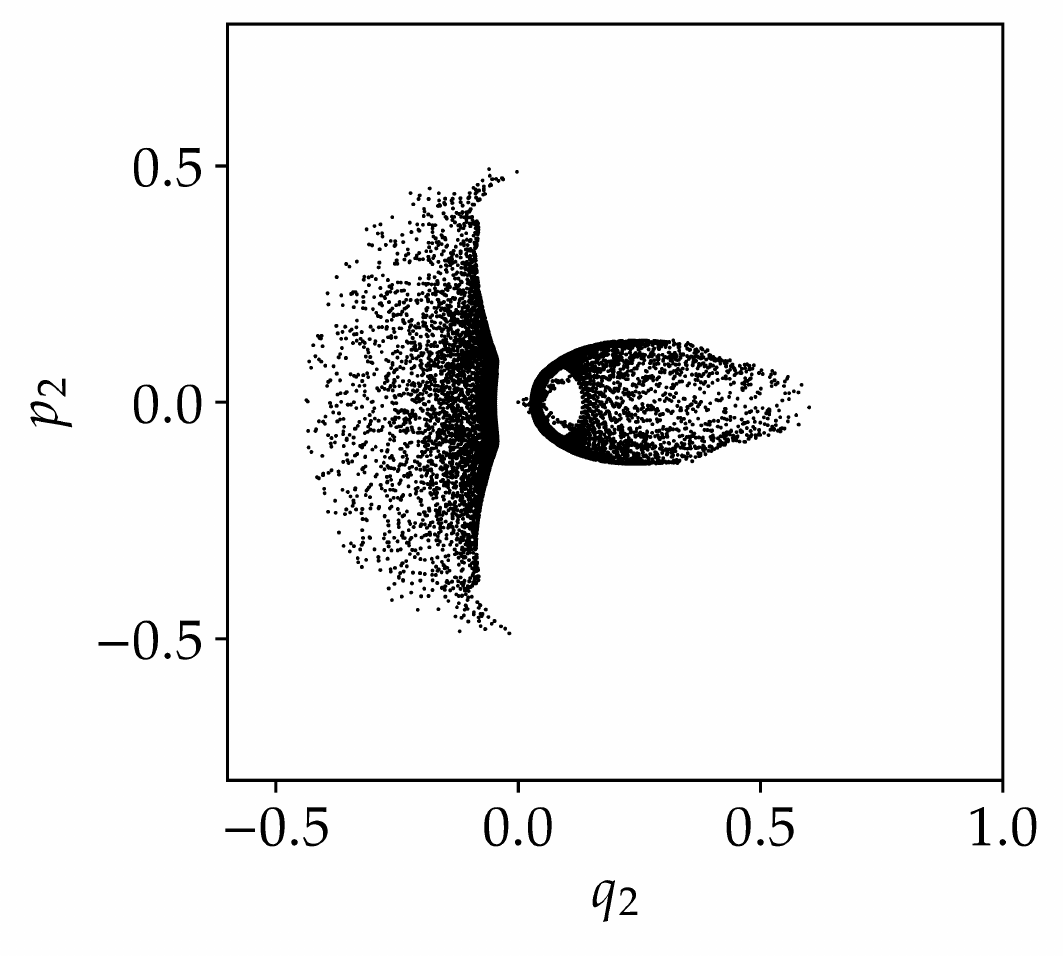}
    \caption{Baseline scheme.}
  \end{subfigure}%
  ~
  \begin{subfigure}[b]{0.49\textwidth}
    \centering
    \includegraphics[width=\textwidth]{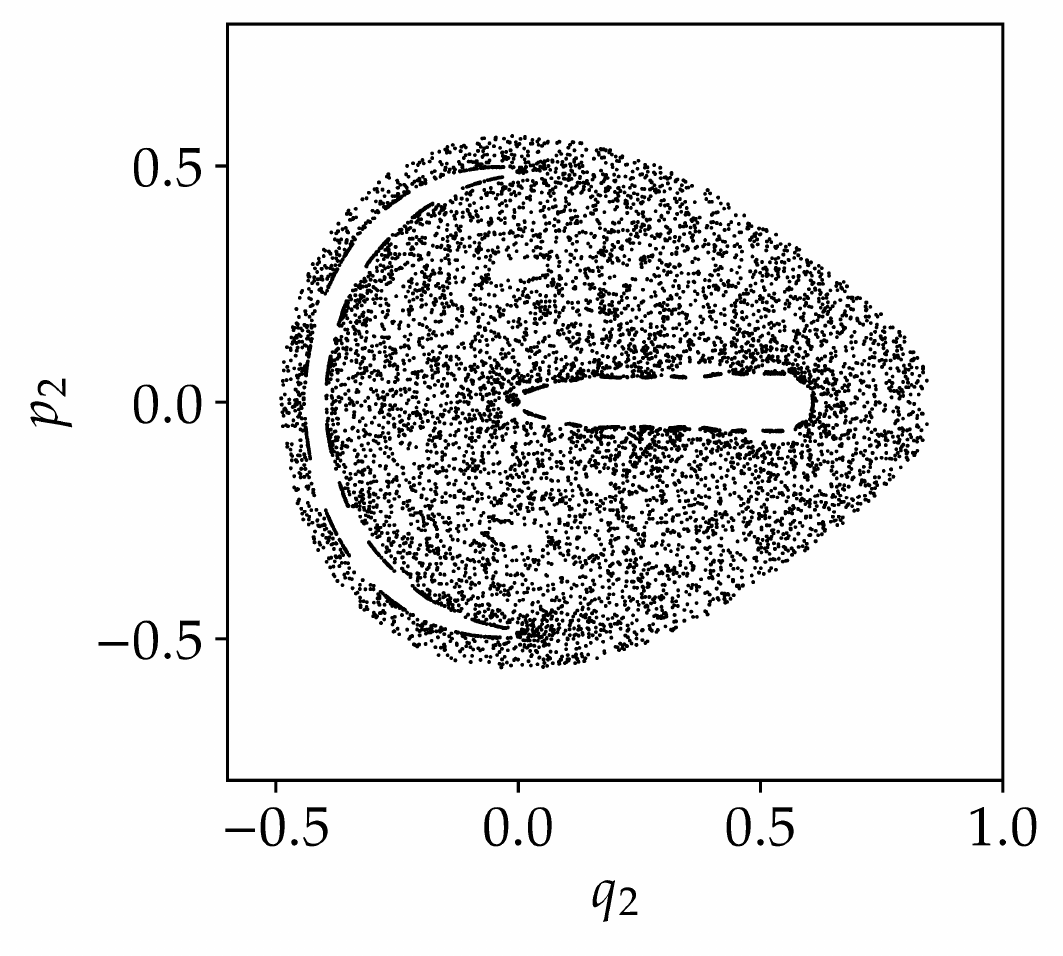}
    \caption{Relaxation scheme.}
  \end{subfigure}%
  \caption{Poincaré sections of numerical solutions of the Hénon-Heiles
           system \eqref{eq:Henon-Heiles} with initial condition
           \eqref{eq:Henon-Heiles-chaotic} using the third order method
           SSPRK(3,3) of \cite{shu1988efficient} with and without relaxation
           and a time step $\dt = 0.1$.}
  \label{fig:Henon-Heiles-chaotic}
\end{figure}

\subsection{Duffing Oscillator}
\label{sec:qualitative-Duffing}

The undamped Duffing oscillator
\begin{equation}
\label{eq:Duffing}
  \od[2]{}{t} q(t) = q(t) - q(t)^3
\end{equation}
can be written as a system of first order ODEs using the momentum
$p(t) = \od{}{t} q(t)$ and is a Hamiltonian system \eqref{eq:Hamiltonian} with total energy
\begin{equation}
  H(q,p) = \frac{1}{2} p^2 - \frac{1}{2} q^2 + \frac{1}{4} q^4.
\end{equation}
The initial condition $(q,p) = (1.4142,0)$ is inside, but very near to, the
separatrix (which crosses $(q,p) = (\sqrt{2}, 0)$) dividing the parts of
phase space that give bounded or unbounded solutions \cite{cano2001comparison}.
The solution should stay on a closed curve near the separatrix in the
half space $q \geq 0$.

Typically, numerical solutions obtained by explicit Runge-Kutta methods
i) lose energy and spiral inwards or ii) cross the separatrix and escape
to the left half plane, as can be seen in Figure~\ref{fig:Duffing}. Of
course, both behaviors are undesired. The relaxation methods
give the correct qualitative behavior.

\begin{figure}[t]
\centering
  \begin{subfigure}[b]{0.4\textwidth}
    \centering
    \includegraphics[width=\textwidth]{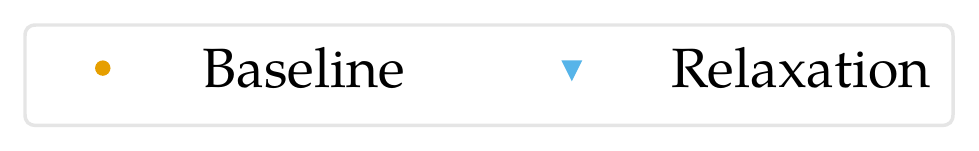}
  \end{subfigure}%
  \\
  \begin{subfigure}[b]{0.49\textwidth}
    \centering
    \includegraphics[width=\textwidth]{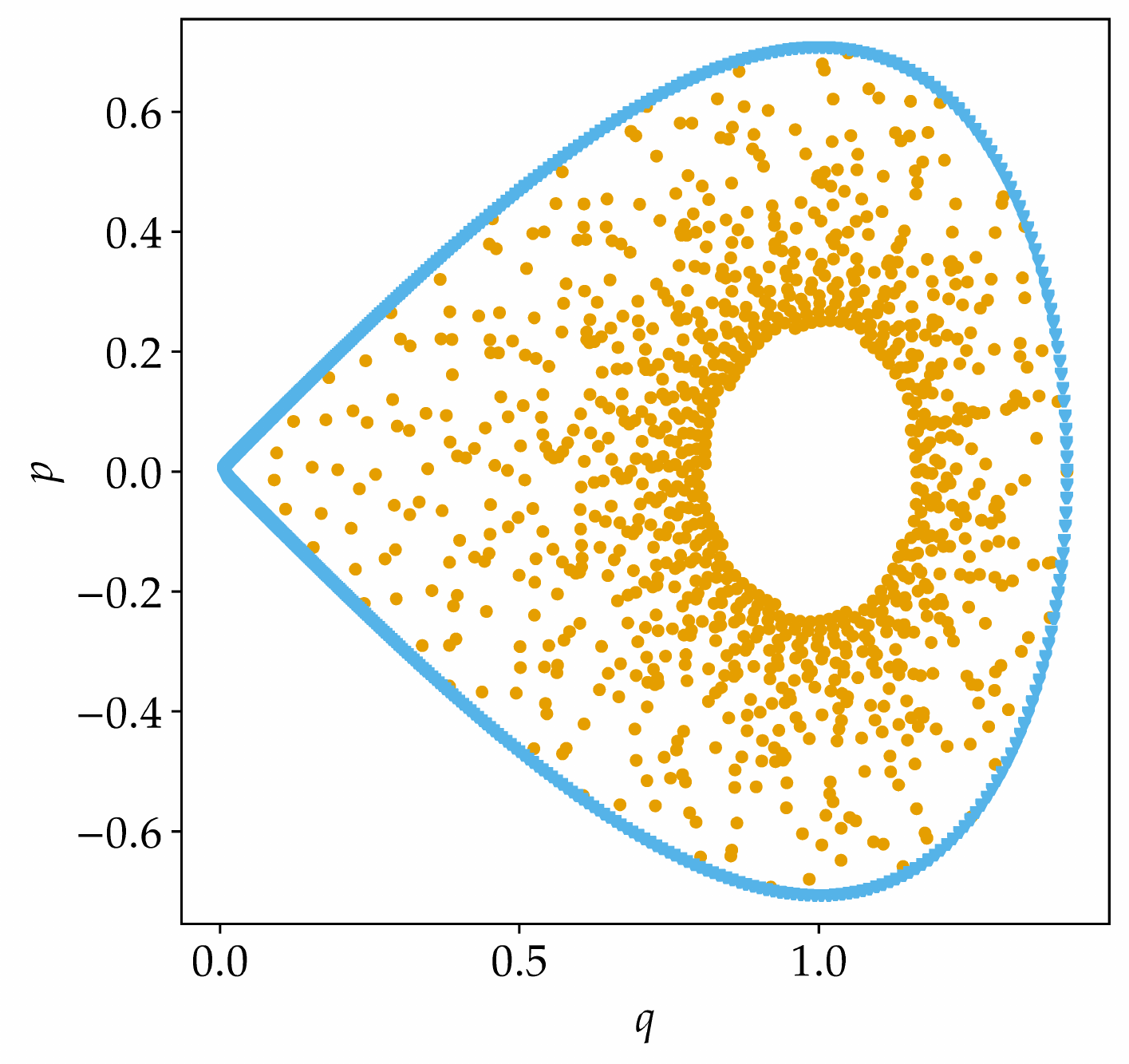}
    \caption{Fourth order method of Kutta~\cite{kutta1901beitrag}, with $\dt=0.5$
    and final time $t=500$.}
  \end{subfigure}%
  ~
  \begin{subfigure}[b]{0.49\textwidth}
    \centering
    \includegraphics[width=\textwidth]{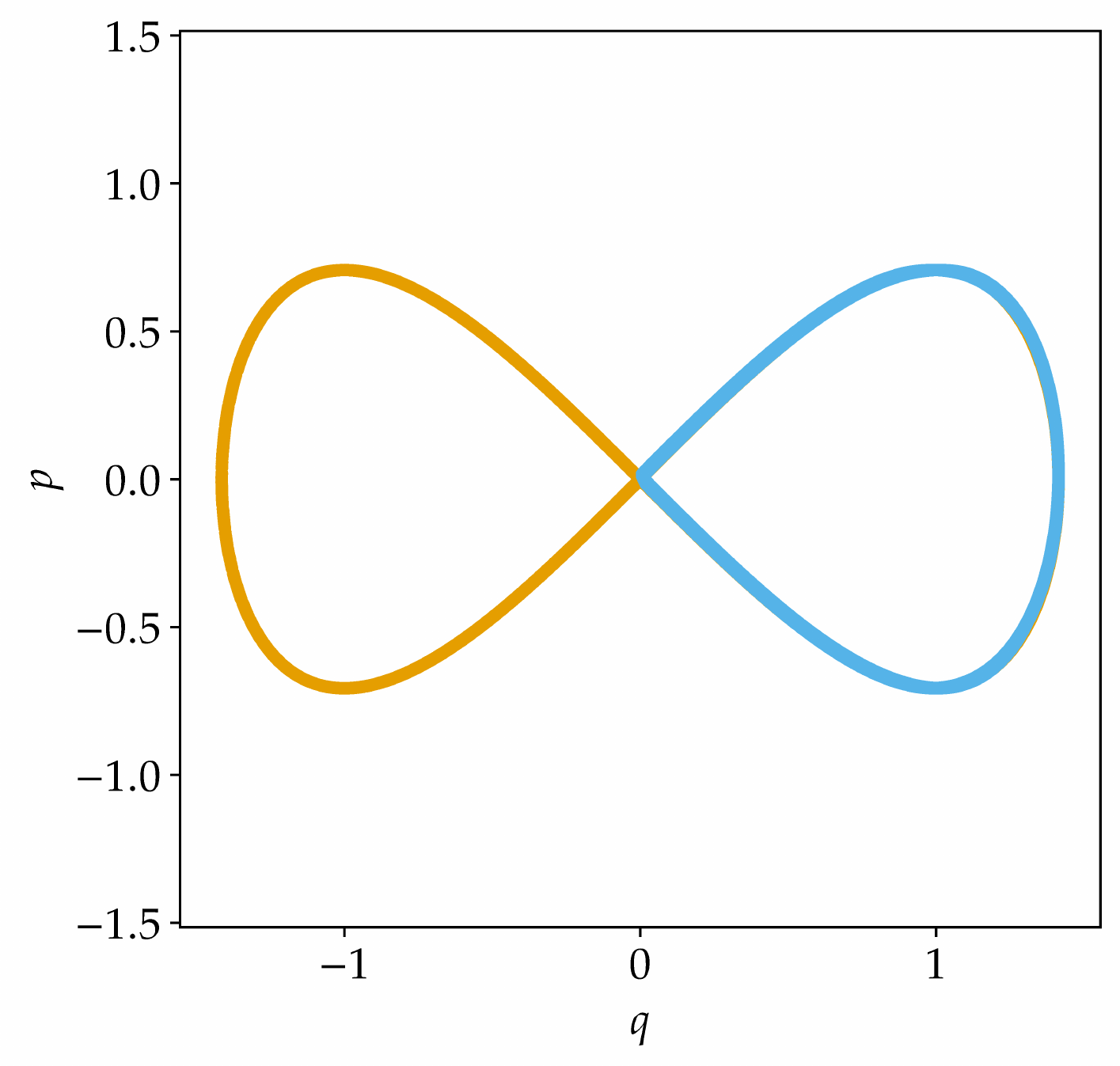}
    \caption{Fifth order method of
             Bogacki \& Shampine~\cite{bogacki1996efficient}, with $\dt=0.25$
             and final time $t=5000$.}
  \end{subfigure}%
  \caption{Simulation of the undamped Duffing oscillator \eqref{eq:Duffing}
           using explicit Runge-Kutta methods with and without relaxation
           to preserve the Hamiltonian.}
  \label{fig:Duffing}
\end{figure}

\subsection{Phase Space Volume}
\label{sec:qualitative-volume}

Symplectic methods are made to preserve the structure of a Hamiltonian
system.  While they do not conserve the Hamiltonian exactly in
general, a modified Hamiltonian is conserved.
Here we investigate whether relaxation Runge-Kutta
methods conserving the Hamiltonian show better conservation of phase
space volume relative to standard Runge-Kutta methods.
While volume preservation is in general not equivalent to symplecticity or
approximate energy conservation \cite{hairer2009energy,hairer2018energy},
the change of phase space volume can be measured approximately for some
models.

At first, the harmonic oscillator
\begin{equation}
\label{eq:harmonic-osc}
  u'(t) = L u(t),
  \quad
  u(0) = \begin{pmatrix} 1 \\ 0 \end{pmatrix},
  \quad
  L =
      \begin{pmatrix}
        0 & -1 \\
        1 & 0 \\
      \end{pmatrix},
\end{equation}
is considered. To measure the change of phase space volume, \num{200}
initial conditions are uniformly randomly chosen in a circle of radius
$0.001$ around $(1, 0)$.
These are evolved in time using the symplectic Euler method as well as
the unmodified and relaxation version of the classical fourth order
Runge-Kutta method with a time step $\dt = 0.25$.
The phase space volume occupied by these points is measured by computing
the volume of the convex hull of these points using the Qhull library
\cite{barber1996quickhull} via SciPy \cite{virtanen2019scipy}. It has been
verified by inspection that this approach is reasonable for the test cases
under consideration.
We choose a symplectic method as reference, since these are constructed to
conserve the phase space volume exactly. Additionally, they conserve quadratic
Hamiltonians such as the energy of the harmonic oscillator. General Hamiltonians
are not conserved exactly by symplectic methods; instead, a modified Hamiltonian is conserved.

The results are shown in Figure~\ref{fig:phase-space-volume-linear-osc}.
Both the symplectic Euler method and the relaxation RK(4,4) method conserve
the phase space volume while the baseline RK(4,4) scheme results in a clear
loss of phase space volume. A linear fit of the relative change of phase
space volume yields the slope \num{1.15e-13} for the relaxation scheme and
\num{1.07e-15} for the symplectic Euler method.

Additionally, two uncoupled harmonic oscillators are considered as a higher
dimensional problem with initial conditions distributed randomly in a cube
of side length $0.001$ around $(1, 0, 0.5, 0.5)$.
Although both oscillators are uncoupled, total energy conservation via
the relaxation approach introduces some coupling.
For this case, the slopes of the linear fits of the relative change of
phase space volume are \num{2.13e-13} for the relaxation RK(4,4) method
and \num{-1.72e-14} for the symplectic Euler method.

\begin{figure}[htb]
\centering
  \begin{subfigure}[b]{0.49\textwidth}
    \centering
    \includegraphics[width=\textwidth]{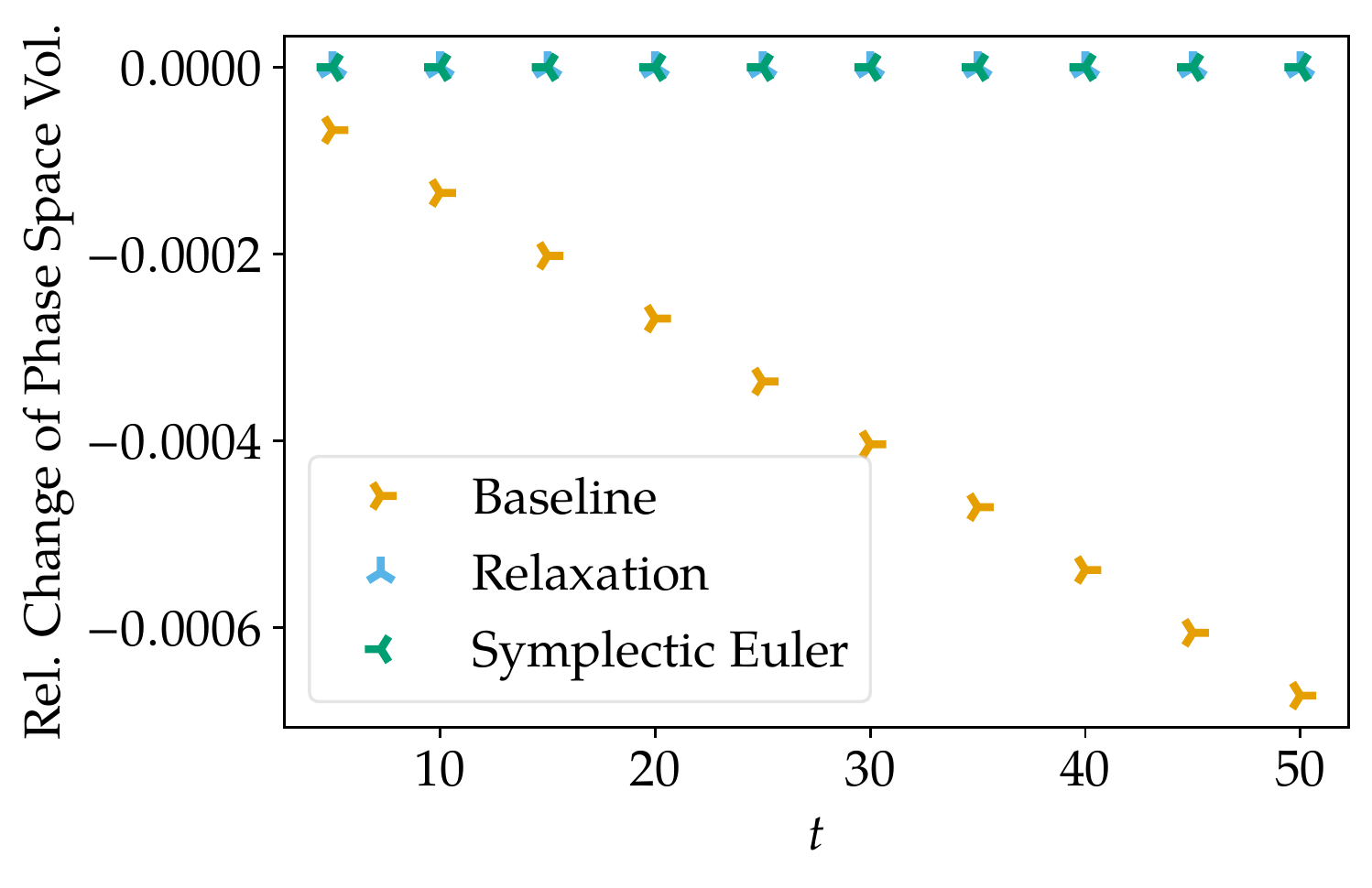}
    \caption{One harmonic oscillator.}
  \end{subfigure}%
  \begin{subfigure}[b]{0.49\textwidth}
    \centering
    \includegraphics[width=\textwidth]{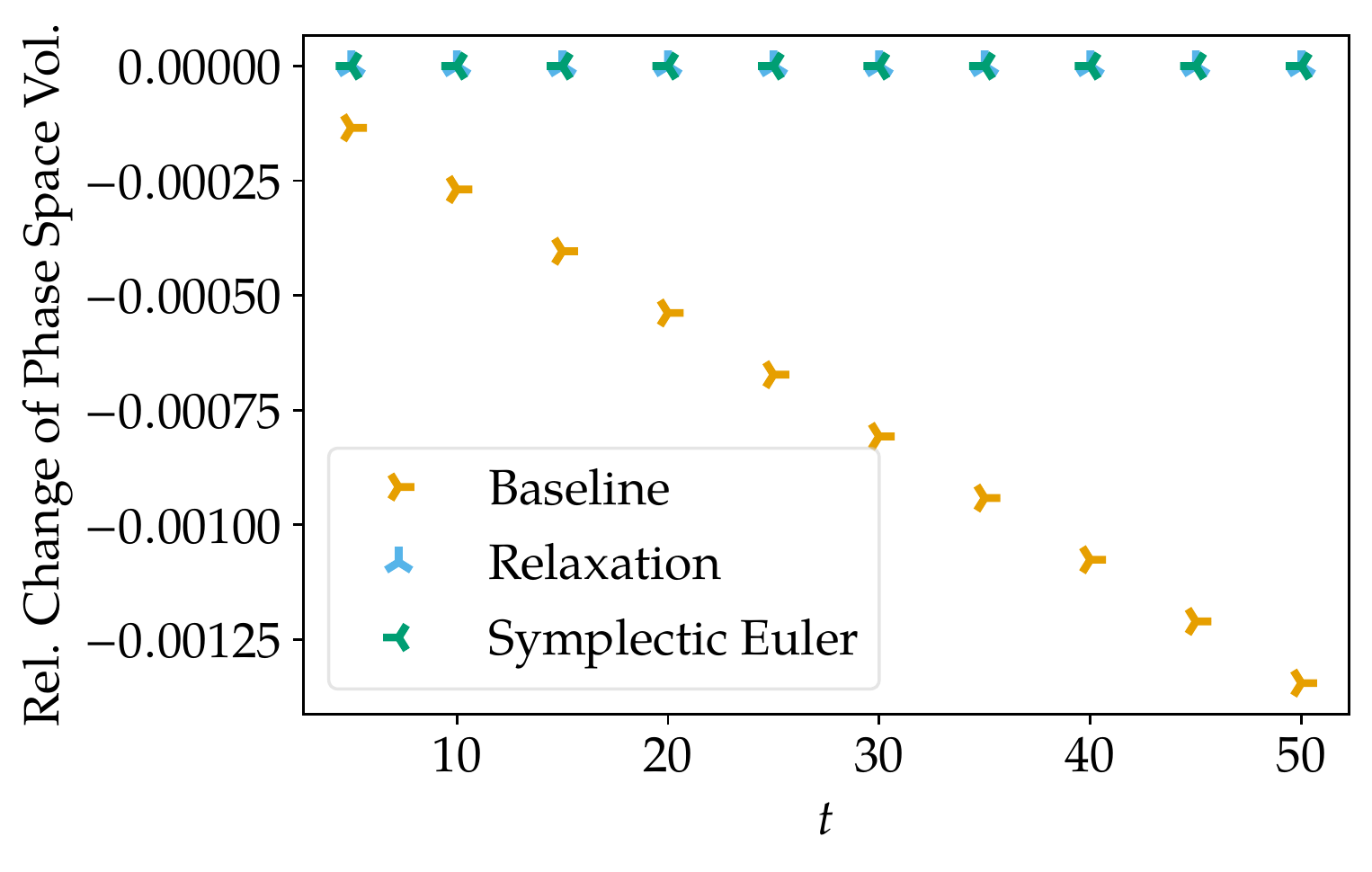}
    \caption{Two harmonic oscillators.}
  \end{subfigure}%
  \caption{Relative change of phase space volume for the harmonic
           oscillator \eqref{eq:harmonic-osc} using the symplectic Euler
           method as well as the classical fourth order Runge-Kutta
           method RK(4,4) with and without relaxation to preserve the
           Hamiltonian/energy.}
  \label{fig:phase-space-volume-linear-osc}
\end{figure}

Additionally, the undamped Duffing oscillator \eqref{eq:Duffing}
is considered as a nonlinear example. The \num{200} initial conditions
are uniformly randomly chosen in a circle of radius $0.001$ around $(1, 0)$.
The relaxation version of RK(4,4) and SSPRK(3,3) yield visually the same
relative change of phas space volume as the symplectic Euler method, while
the corresponding baseline schemes result in a clear change of phase space
volume, as can be seen in Figure~\ref{fig:phase-space-volume-duffing-osc}.
While the appropriateness of the approach to measure the phase space volume
has still been verified by inspection, it becomes a bit less exact. In
particular, linear fits of the relative change of phase space volume for both
relaxation and symplectic methods yield a slope of approximately \num{1e-8}
instead of machine accuracy. For more complicated problems, this approach to
measure the phase space volume does not seem to be appropriate.

\begin{figure}
\centering
  \begin{subfigure}[b]{0.49\textwidth}
    \centering
    \includegraphics[width=\textwidth]{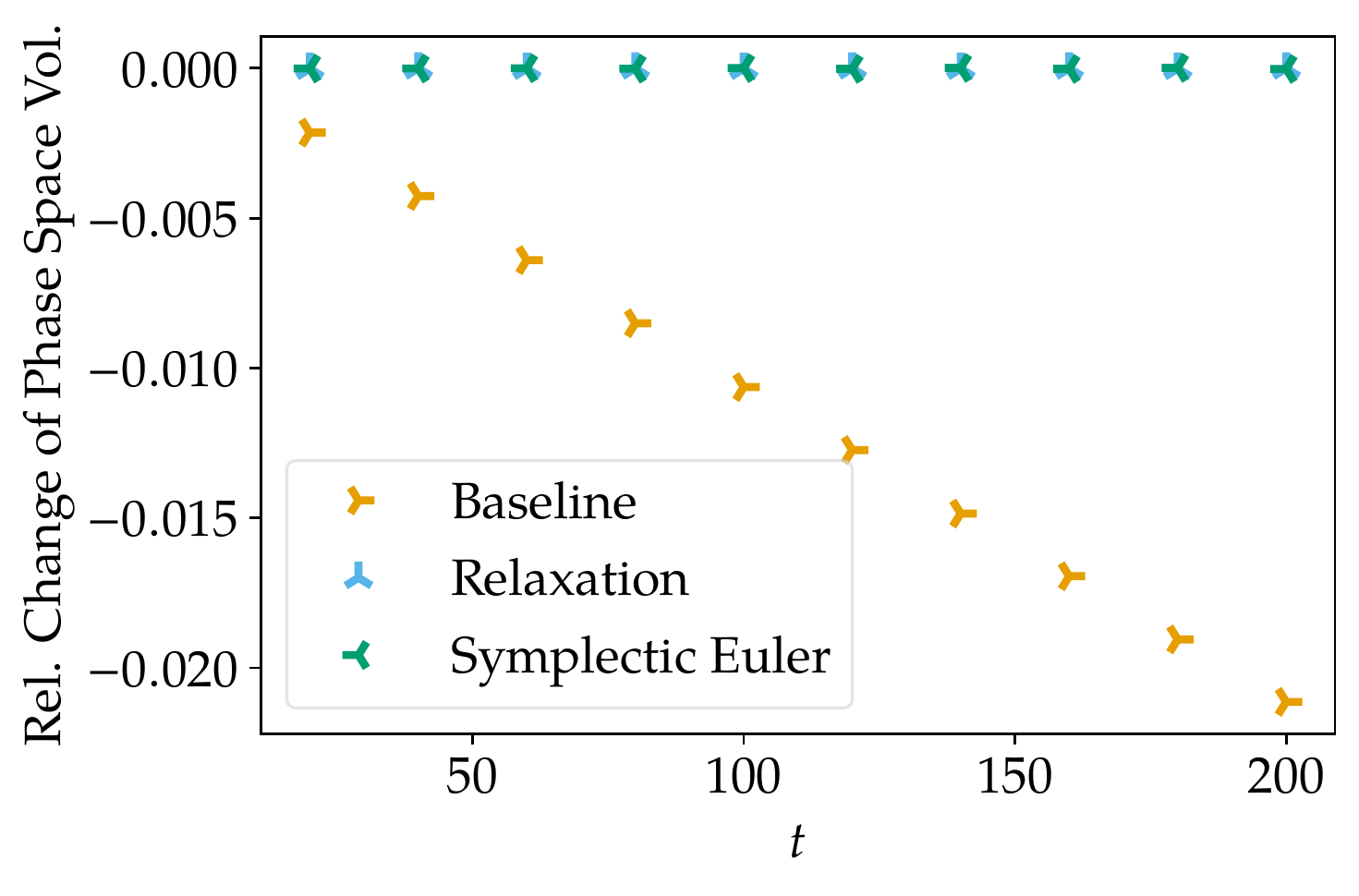}
    \caption{RK(4,4).}
  \end{subfigure}%
  \begin{subfigure}[b]{0.49\textwidth}
    \centering
    \includegraphics[width=\textwidth]{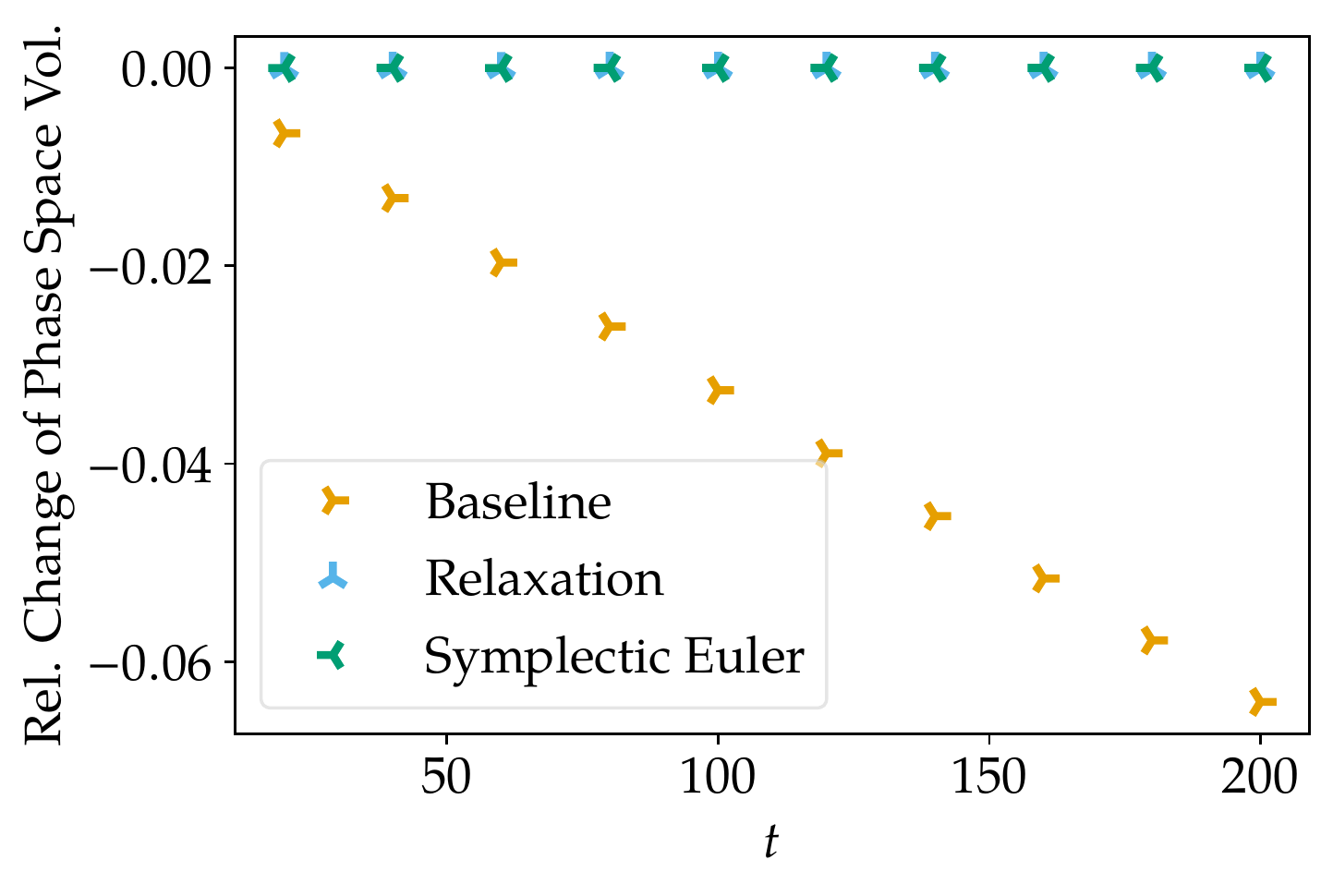}
    \caption{SSPRK(3,3).}
  \end{subfigure}%
  \caption{Relative change of phase space volume for the Duffing
           oscillator \eqref{eq:Duffing} using the symplectic Euler
           method as well as RK(4,4) and SSPRK(3,3) with and without
           relaxation to preserve the Hamiltonian.}
  \label{fig:phase-space-volume-duffing-osc}
\end{figure}

\section{Accuracy}
\label{sec:accuracy}

In this section we study the impact of relaxation on accuracy for Hamiltonian
problems. It has been observed (and proven in an asymptotic sense) that
relaxation does not reduce the order of accuracy of a RK method.
Here we show that it can in fact improve the accuracy for Hamiltonian
problems. In Section \ref{sec:error-growth} we demonstrate that RRK methods
reduce the error growth in time for Hamiltonian problems.
In Section \ref{sec:superconvergence}
we show that relaxation increases the convergence rate by one for odd
order methods applied to a certain class of Hamiltonian problems.

\subsection{Error Growth for Oscillatory Problems}
\label{sec:error-growth}

Structure preserving methods, especially symplectic or energy conserving
ones, result in a good behavior of the numerical error for periodic
(Hamiltonian) problems, cf.
\cite{calvo1993development,portillo1995lack,cano1997error,calvo2011error}.
For periodic problems where the period is conserved by a relaxation
approach, e.g. if the period depends only on the Hamiltonian and the
RRK method conserves the Hamiltonian, the asymptotic error growth
is linear, even if adaptive time steps are used. This can be proved
by applying Theorem~5.1 of \cite{calvo2011error}, noticing that the
direction
\begin{equation}
  d = \sum_{i=1}^s b_i f_i
\end{equation}
remains bounded away from zero for $\dt \to 0$, unless the system is stationary.
In contrast, methods that do not preserve some structure of the
system (either exactly or to some appropriate order of accuracy) will in
general result in an asymptotic error growth that is linear at first
and quadratic after some time, usually after the first period, cf.
\cite{calvo1993development,portillo1995lack,cano1997error,calvo2011error}.

\subsubsection{Nonlinear Oscillator}
\label{sec:error-growth-nonlinear-osc}

For the nonlinear oscillator
\begin{equation}
\label{eq:nonlinear-osc}
  \od{}{t} \begin{pmatrix} u_1(t) \\ u_2(t) \end{pmatrix}
  =
  \norm{u(t)}^{-2} \begin{pmatrix} -u_2(t) \\ u_1(t) \end{pmatrix},
  \quad
  u^0
  =
  \begin{pmatrix} 1 \\ 0 \end{pmatrix},
\end{equation}
of \cite{ranocha2020strong,ranocha2019energy},
the period is a function of the energy
of the solution. Hence, the theory predicts that the error of the baseline
Runge-Kutta schemes grows linearly at first and quadratically afterwards
while relaxation methods should have a linear error growth in time.
The numerical results shown in Figure~\ref{fig:error-growth-nonlinear-osc}
confirm this theory for the third order method of Heun~\cite{heun1900neue},
the fourth order method of Fehlberg~\cite[Table~III]{fehlberg1969low},
and the fifth order method of Bogacki \& Shampine~\cite{bogacki1996efficient}.

\begin{figure}
\centering
  \begin{subfigure}{0.32\textwidth}
    \centering
    \includegraphics[width=\textwidth]{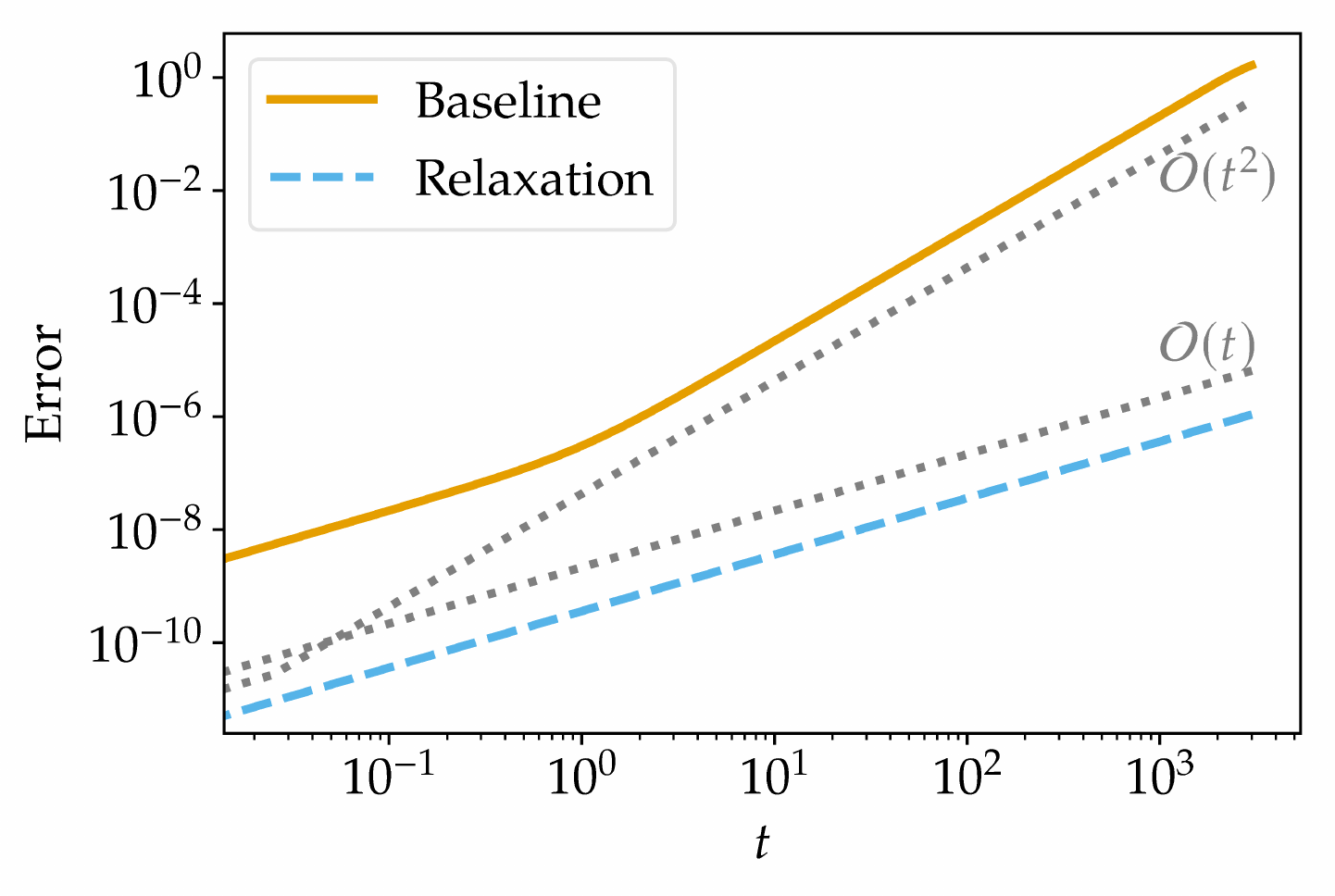}
    \caption{Third order method of Heun~\cite{heun1900neue},
             $\dt = 0.025$.}
  \end{subfigure}%
  ~
  \begin{subfigure}{0.32\textwidth}
    \centering
    \includegraphics[width=\textwidth]{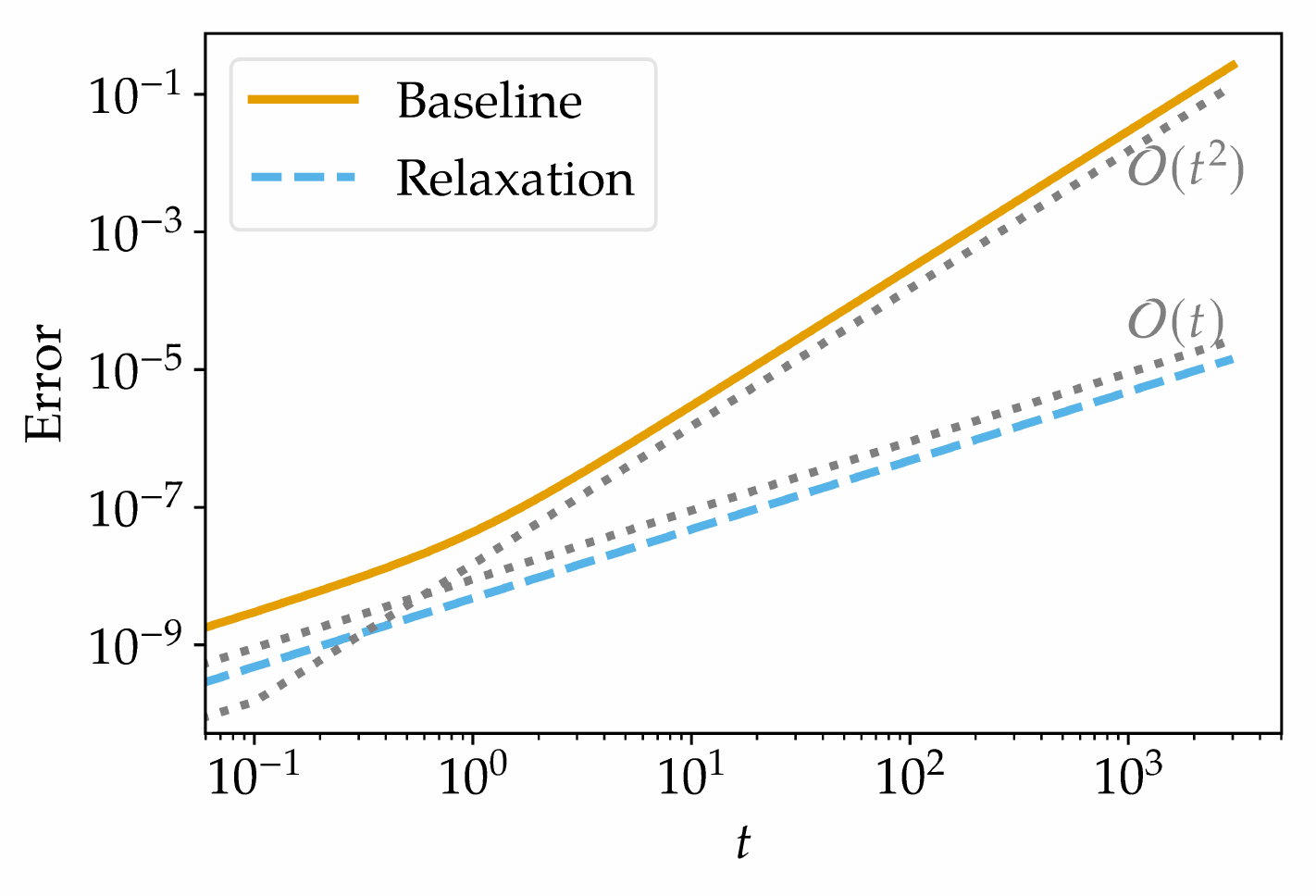}
    \caption{Fourth order method of
             Fehlberg~\cite[Table~III]{fehlberg1969low},
             $\dt = 0.1$.}
  \end{subfigure}%
  ~
  \begin{subfigure}{0.32\textwidth}
    \centering
    \includegraphics[width=\textwidth]{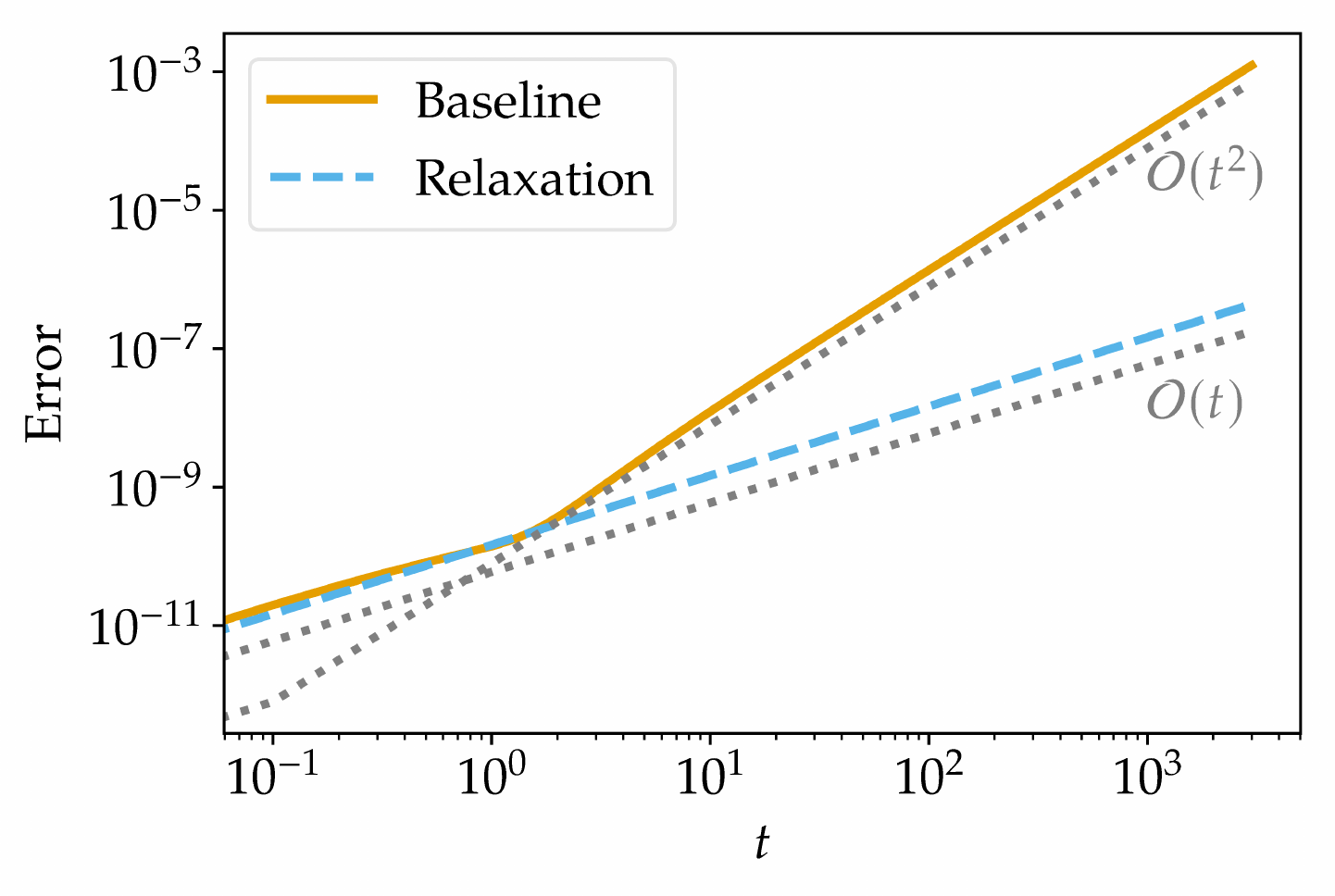}
    \caption{Fifth order method of
             Bogacki \& Shampine~\cite{bogacki1996efficient},
             $\dt = 0.1$.}
  \end{subfigure}%
  \caption{Error growth in time for the nonlinear oscillator
            \eqref{eq:nonlinear-osc} using different baseline and
            relaxation Runge-Kutta methods.}
  \label{fig:error-growth-nonlinear-osc}
\end{figure}

\subsubsection{Kepler Problem}
\label{sec:error-growth-Kepler}

The Kepler problem
\begin{equation}
\label{eq:Kepler}
\begin{gathered}
  \od{}{t} q(t)
  = \od{}{t} \begin{pmatrix} q_1(t) \\ q_2(t) \end{pmatrix}
  = p(t),
  \quad
  \od{}{t} p_i(t) = - \frac{q_i(t)}{ \abs{q(t)}^3 },
  \\
  q(0) = \begin{pmatrix} 1-e \\ 0 \end{pmatrix},
  \quad
  p(0) = \begin{pmatrix} 0 \\ \sqrt{ (1-e) / (1+e) } \end{pmatrix},
\end{gathered}
\end{equation}
with eccentricity $e = 0.5$ is a Hamiltonian system \eqref{eq:Hamiltonian} with Hamiltonian
\begin{equation}
  H(q,p) = \frac{1}{2} \abs{p}^2 - \frac{1}{\abs{q}},
\end{equation}
where the angular momentum
\begin{equation}
  L(q,p) = q_1 p_2 - q_2 p_1
\end{equation}
is an additional conserved quantity, cf.
\cite[Section~1.2.4]{sanzserna1994numerical,calvo1993development}.
The period depends only on the Hamiltonian, which can be conserved by
relaxation Runge-Kutta methods. Another possibility is to employ the
relaxation approach to conserve the angular momentum.

Results of baseline schemes, RRK methods conserving the Hamiltonian,
and RRK methods conserving the angular momentum are shown in
Figure~\ref{fig:error-growth-Kepler}. As expected, the asymptotic
error growth is quadratic whenever the Hamiltonian is not conserved
while the energy conserving RRK methods result in a linear error growth
in time, similarly to symplectic schemes.
\begin{figure}
\centering
  \begin{subfigure}{0.32\textwidth}
    \centering
    \includegraphics[width=\textwidth]{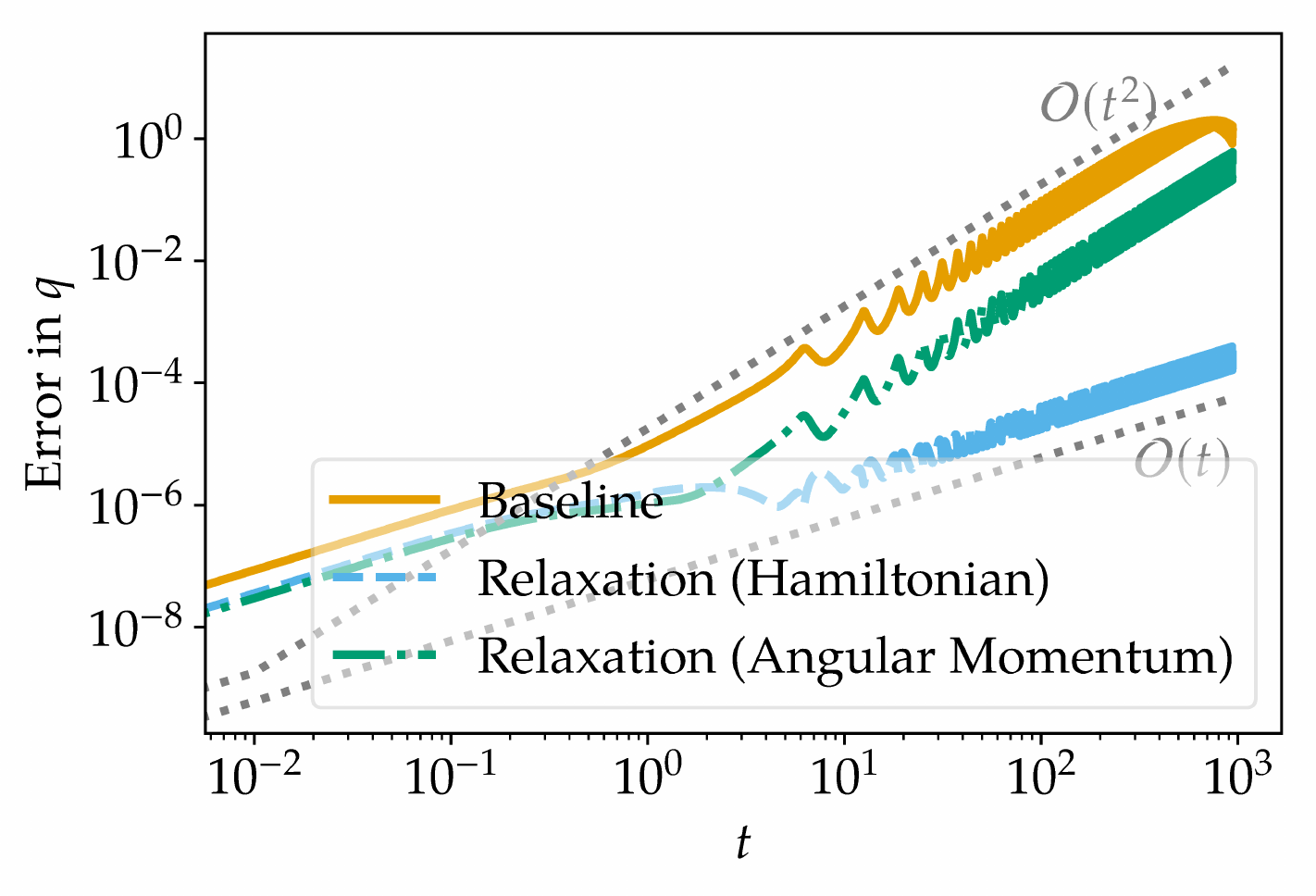}
    \caption{Third order method of Shu \& Osher~\cite{shu1988efficient},
              $\dt = 0.001$.}
  \end{subfigure}%
  ~
  \begin{subfigure}{0.32\textwidth}
    \centering
    \includegraphics[width=\textwidth]{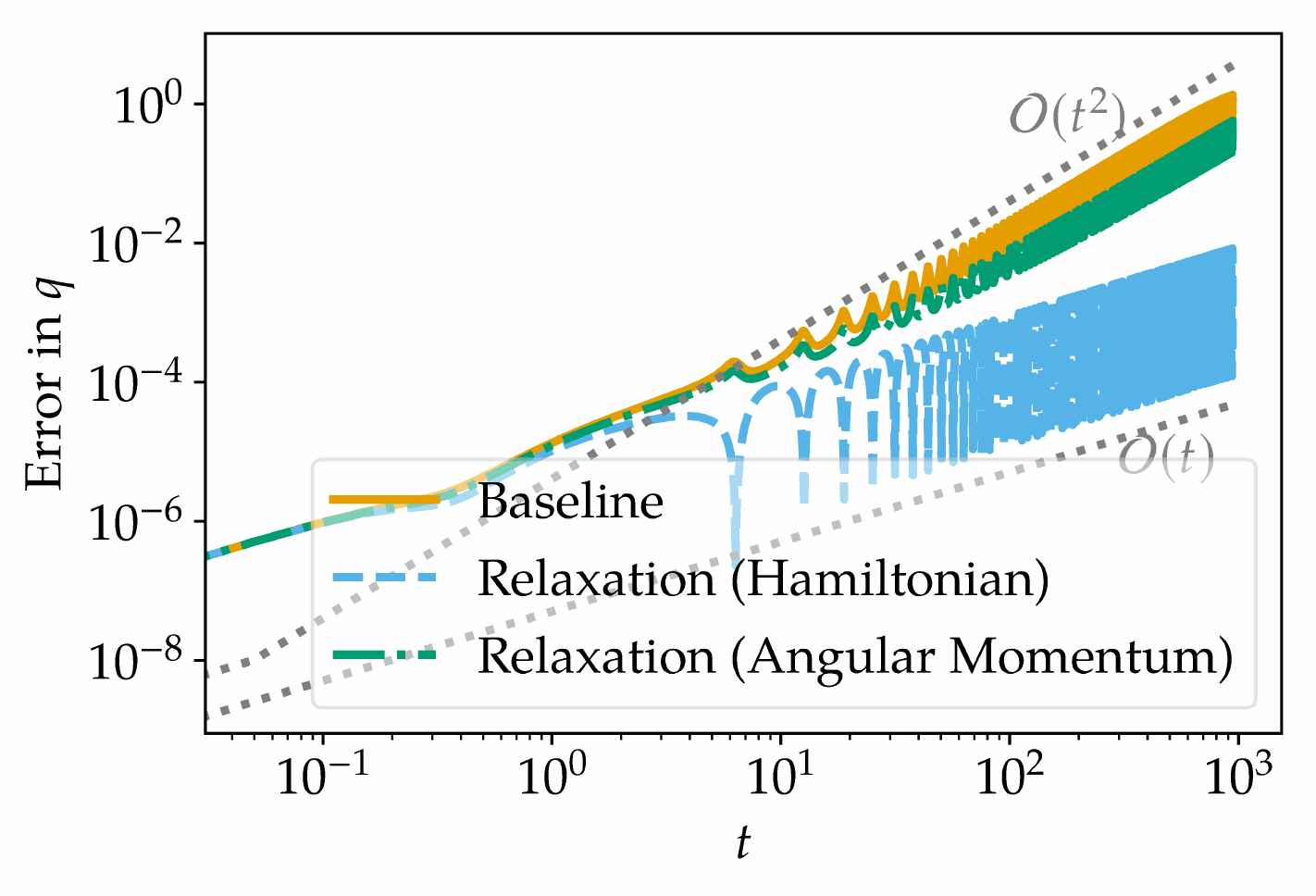}
    \caption{Fourth order method of Kutta~\cite{kutta1901beitrag},
              $\dt = 0.05$.}
  \end{subfigure}%
  ~
  \begin{subfigure}{0.32\textwidth}
    \centering
    \includegraphics[width=\textwidth]{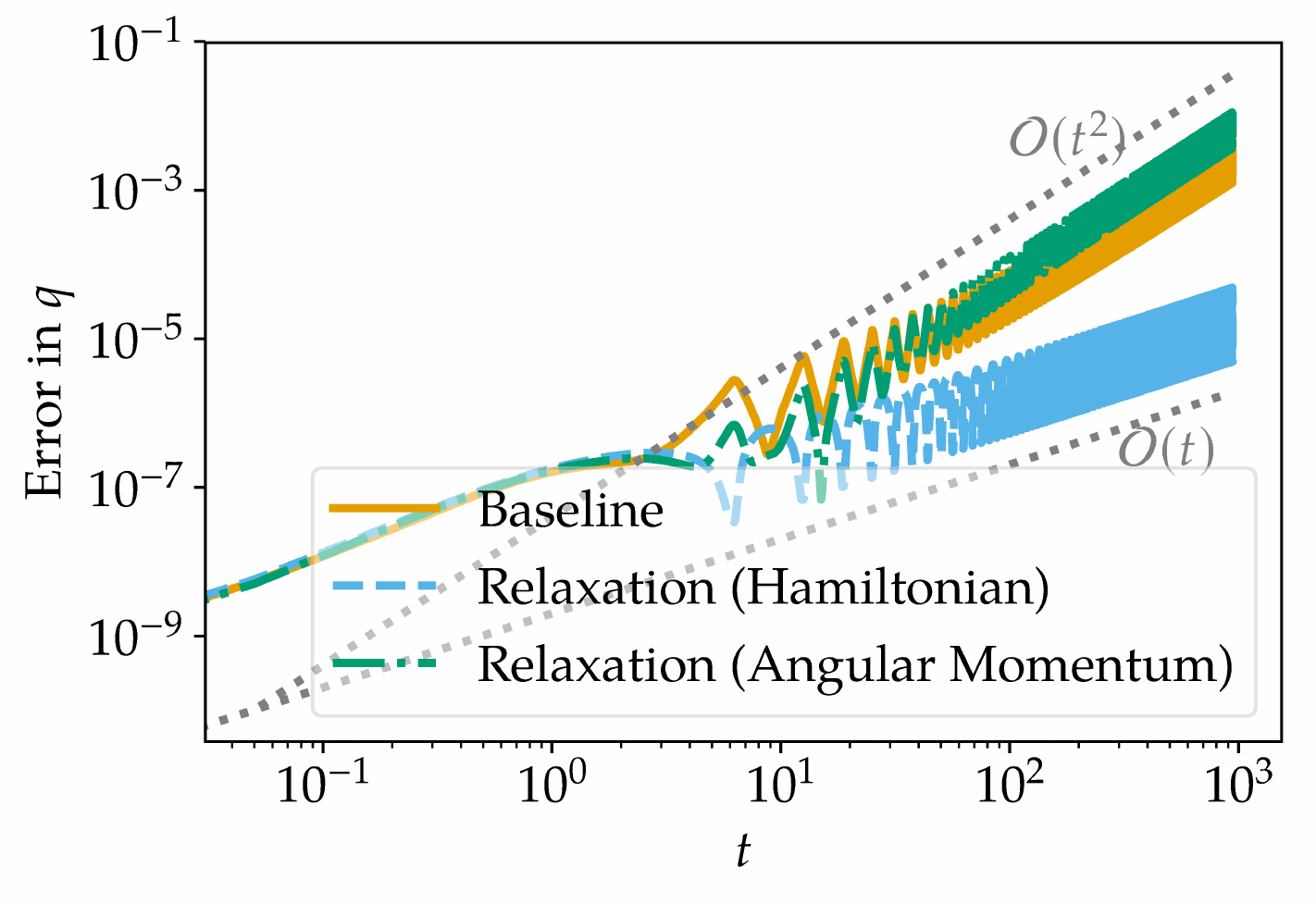}
    \caption{Fifth order method of Prince \& Dormand~\cite{prince1981high},
              $\dt = 0.05$.}
  \end{subfigure}%
  \caption{Error growth in time for the Kepler problem
            \eqref{eq:Kepler} using different baseline and
            relaxation Runge-Kutta methods.}
  \label{fig:error-growth-Kepler}
\end{figure}

\subsubsection{Korteweg-de Vries Equation}
\label{sec:KdV}

The Korteweg-de Vries (KdV) equation
\begin{equation}
\label{eq:KdV}
  \partial_t u + \partial_x \frac{u^2}{2} + \partial_x^3 u = 0
\end{equation}
is well-known in the literature as a nonlinear PDE which admits soliton
solutions of the form
\begin{equation}
  u(t,x) = A \cosh\bigl( \sqrt{3 A} (x - c t - \mu) / 6 \bigr)^{-2},
  \quad
  c = \nicefrac{A}{3},
\end{equation}
where $A$ is the amplitude, $c$ the wave speed, and $\mu$ an
arbitrary constant. The KdV equation possesses an infinite hierarchy
of conserved quantities, including the mass $\int u$ and the energy
$\frac{1}{2} \norm{u}^2 = \frac{1}{2} \int u^2$.

Using periodic boundary conditions, mass- and energy-conservative
semidiscretizations using central finite difference or Fourier collocation
schemes can be constructed using a split form as
\begin{equation}
\label{eq:KdV-semidiscrete}
  \partial_t u + \frac{1}{3} \bigl( D_1 (u \cdot u) + u \cdot (D_1 u) \bigr) + D_3 u = 0,
\end{equation}
where $D_k$ is the discrete derivative operator approximating the $k$th
derivative, i.e. $D_k u \approx \partial_x^k u$. Here $u$ is the vector of
solution values at the spatial grid points, and $\cdot$ denotes elementwise
multiplication.
The semidiscretization \eqref{eq:KdV-semidiscrete} is obtained by averaging
the conservative form $\partial_x \nicefrac{u^2}{2}$ of the nonlinear term and
the equivalent form $u \partial_x u$. Such split forms are often necessary to
obtain conservative or dissipative semidiscretizations of nonlinear PDEs since
the product and chain rules do not hold discretely \cite{ranocha2019mimetic}.
Although the idea to use split forms is not new
\cite[eq. (6.40)]{richtmyer1967difference}, it is still state of the art
\cite{gassner2016split,ranocha2018thesis}.

In the context of structure preserving schemes, especially energy
conservative schemes, the KdV equation has been studied in
\cite{sanz1982explicit,sanz1983method,frutos1997accuracy}.
In particular, the error growth in time for soliton solutions has been
studied in \cite{frutos1997accuracy}, where it is proved that mass- and energy-conservative
discretizations result in an asymptotically linear error growth while
other discretizations yield a quadratically growing error in time.

To study the behavior of relaxation Runge-Kutta methods for the KdV
equation, the domain $[-20, 60]$ is discretized using a pseudospectral
Fourier collocation method \eqref{eq:KdV-semidiscrete} with $N = 256$
nodes. With this semi-discretization, it is expected
that the temporal error will be dominant unless very small time steps are used.
Since the third derivative in space results in a
stiff semi-discrete problem, implicit Runge-Kutta methods are advantageous.
Here, the two stage, third order SDIRK method of Nørsett \cite{norsett1974semiexplicit}
given in \cite[Table~II.7.2]{hairer2008solving} is used with a time step
$\dt = 0.5$ to advance till the final time $t = 600$.
The stage equations are solved using \texttt{fsolve} from SciPy
\cite{virtanen2019scipy}.
This SDIRK method has
also been used as an example in \cite{frutos1997accuracy} (without relaxation
or orthogonal projection). The initial
condition is given as the projection of the soliton with amplitude
$A = 2$ and offset $\mu = 40$. We also test the behavior of this
discretization when orthogonal projection is applied (instead of relaxation)
in order to preserve energy.

As can be seen in Figure~\ref{fig:KdV-error-and-energy}, the relaxation and
projection methods conserve the total energy up to small errors (due to roundoff
and errors in the solution of the projection equation) while the baseline
SDIRK method results in a decaying energy. The relaxation method, due
to its linear covariance, also preserves the mass, while the projection method
does not (indeed, over the course of the simulation the projection method leads
to a mass increase of almost \SI{10}{\percent}). In accordance with the theory
of \cite{frutos1997accuracy}, the error of the baseline scheme grows
quadratically in time until it saturates at \SI{100}{\percent} error.
In contrast, the energy conservative RRK method
yields a linearly growing error in time.
This results in a much smaller error at the final time, yielding a
numerical solution that is visually indistinguishable from the analytical
one, as shown in Figure~\ref{fig:KdV-solution}.
The projection method also exhibits linear growth for some time,
but at very long times shows quadratic growth due to a noticeable
phase error in Figure~\ref{fig:KdV-solution}. This is again in
accordance with the theory of \cite{frutos1997accuracy}, since the total
mass is not conserved.
Similar results have been obtained for the fourth order SDIRK methods
SDIRK(3,4) of \cite[Table~IV.6.5]{hairer2010solving} and SDIRK(5,4) of
\cite[eq. (6.18)]{hairer2010solving}.

\begin{figure}
\centering
  \begin{subfigure}[t]{0.6\textwidth}
    \centering
    \includegraphics[width=\textwidth]{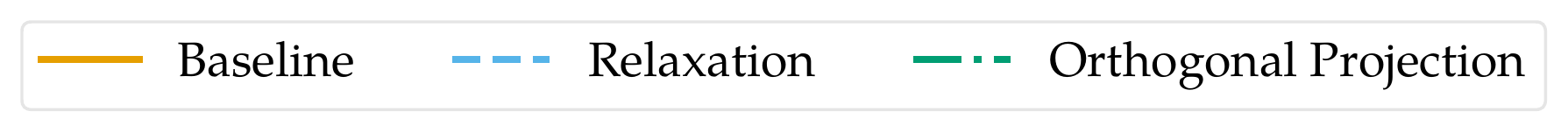}
  \end{subfigure}%
  \\
  \begin{subfigure}[t]{0.32\textwidth}
    \centering
    \includegraphics[width=\textwidth]{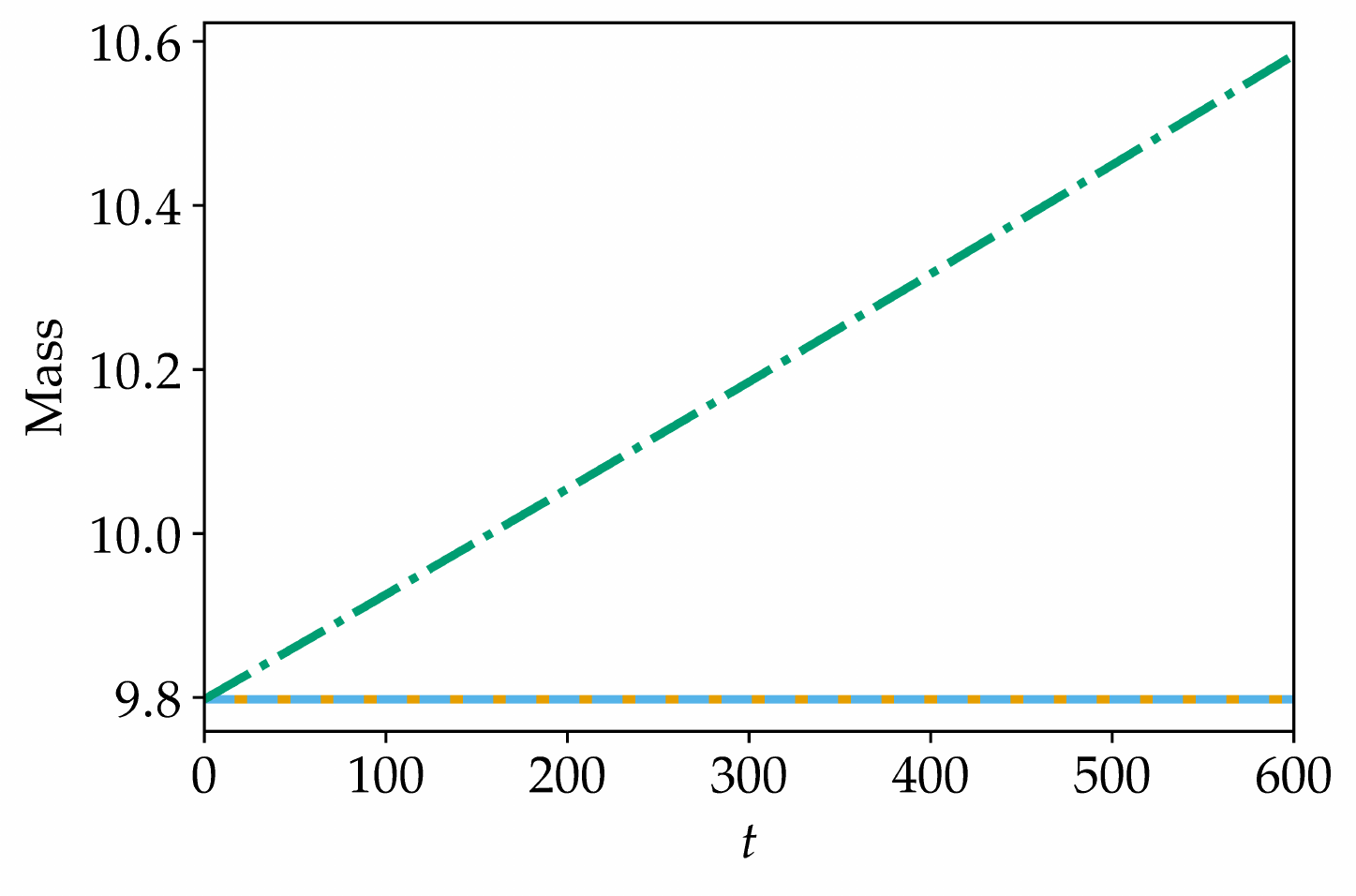}
    \caption{Mass of numerical solutions.}
  \end{subfigure}%
  ~
  \begin{subfigure}[t]{0.32\textwidth}
    \centering
    \includegraphics[width=\textwidth]{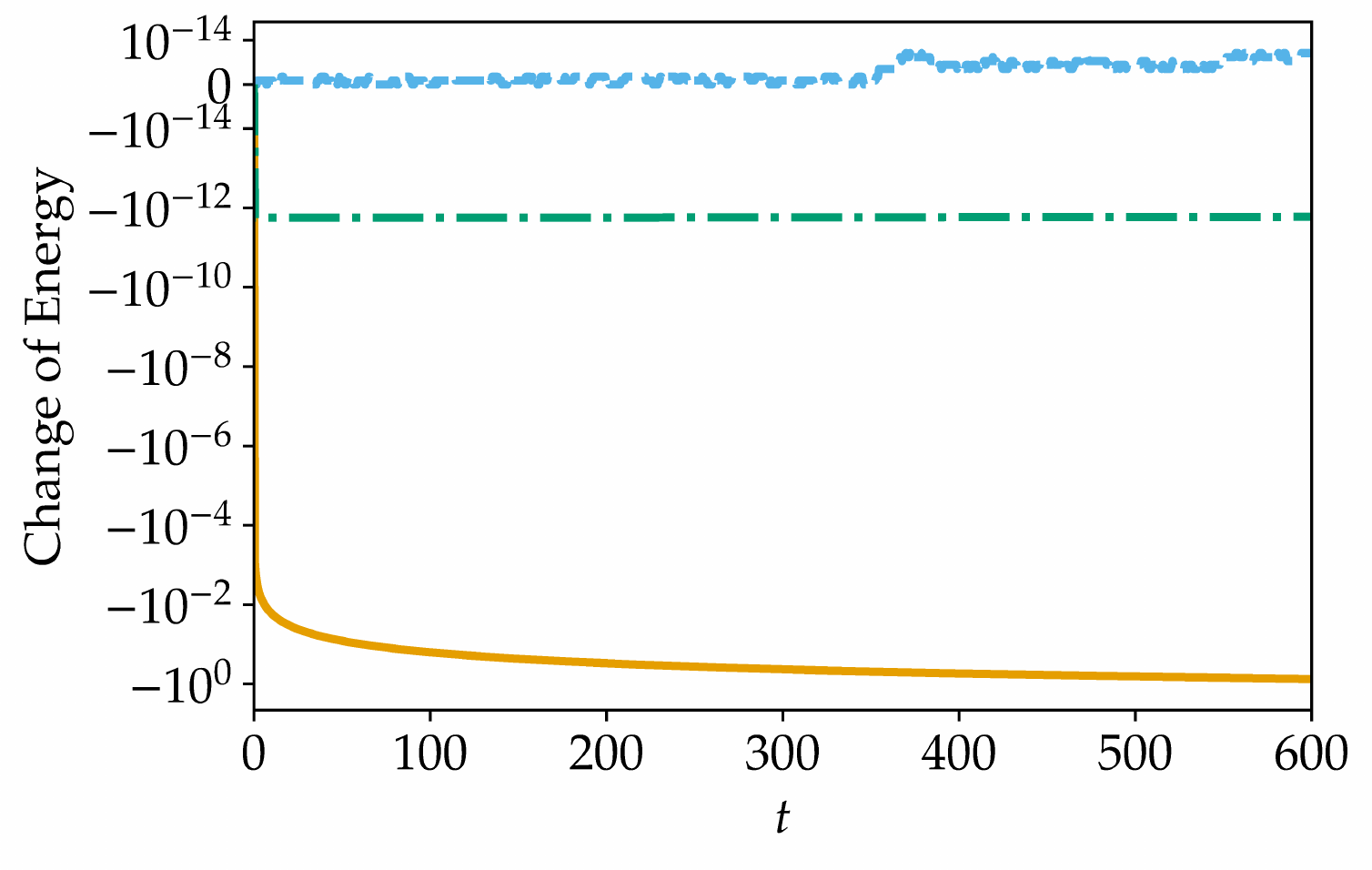}
    \caption{Energy of numerical solutions.}
  \end{subfigure}%
  ~
  \begin{subfigure}[t]{0.32\textwidth}
    \centering
    \includegraphics[width=\textwidth]{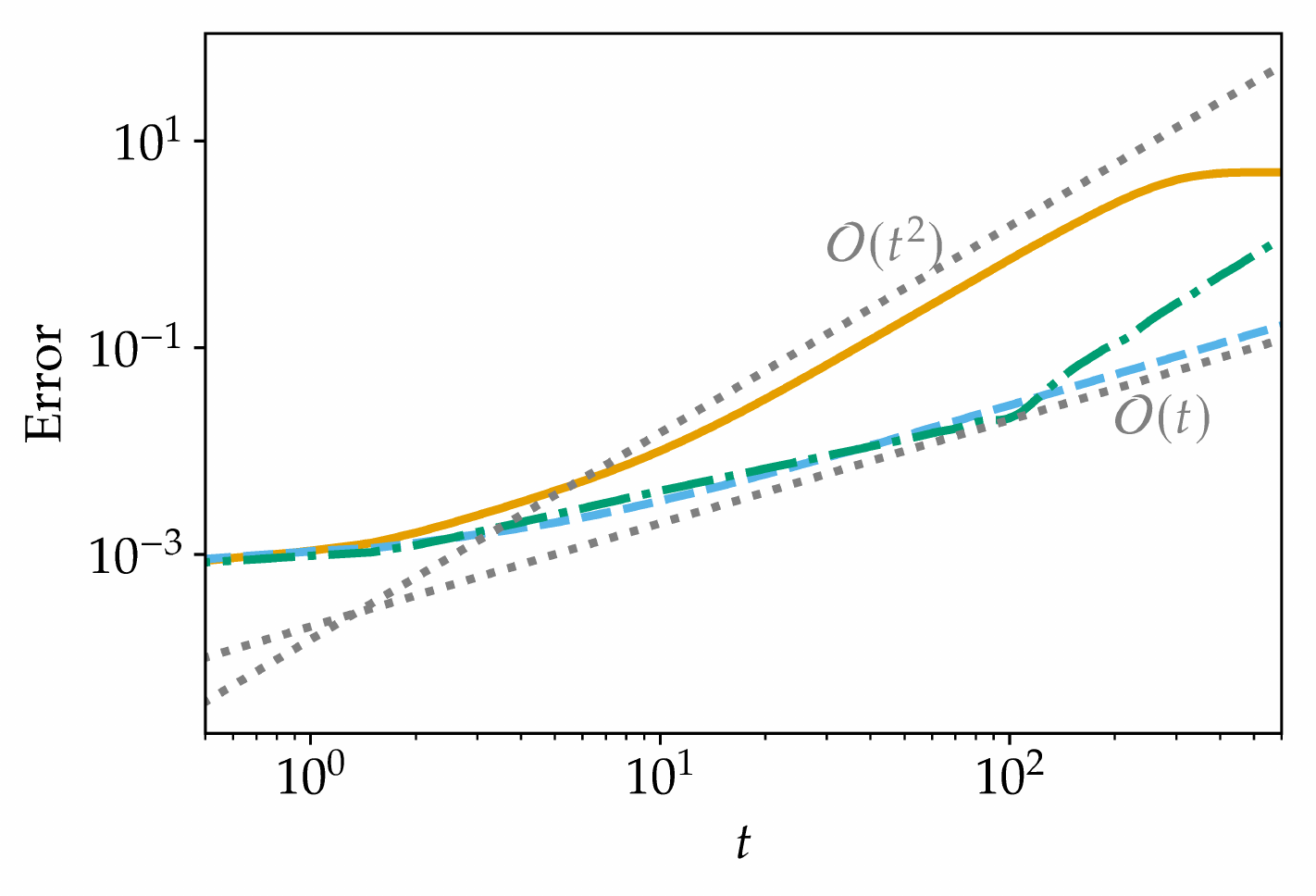}
    \caption{Error of numerical solutions.}
  \end{subfigure}%
  \caption{Total masses \& energies and errors of numerical solutions with
           and without relaxation to preserve the energy for the KdV equation
           \eqref{eq:KdV-semidiscrete}.}
  \label{fig:KdV-error-and-energy}
\end{figure}

\begin{figure}
\centering
  \begin{subfigure}[t]{0.7\textwidth}
    \centering
    \includegraphics[width=\textwidth]{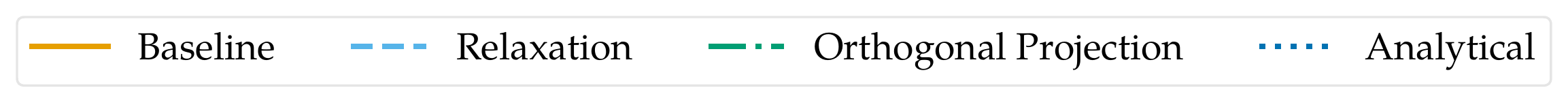}
  \end{subfigure}%
  \\
  \begin{subfigure}[b]{0.49\textwidth}
    \centering
    \includegraphics[width=\textwidth]{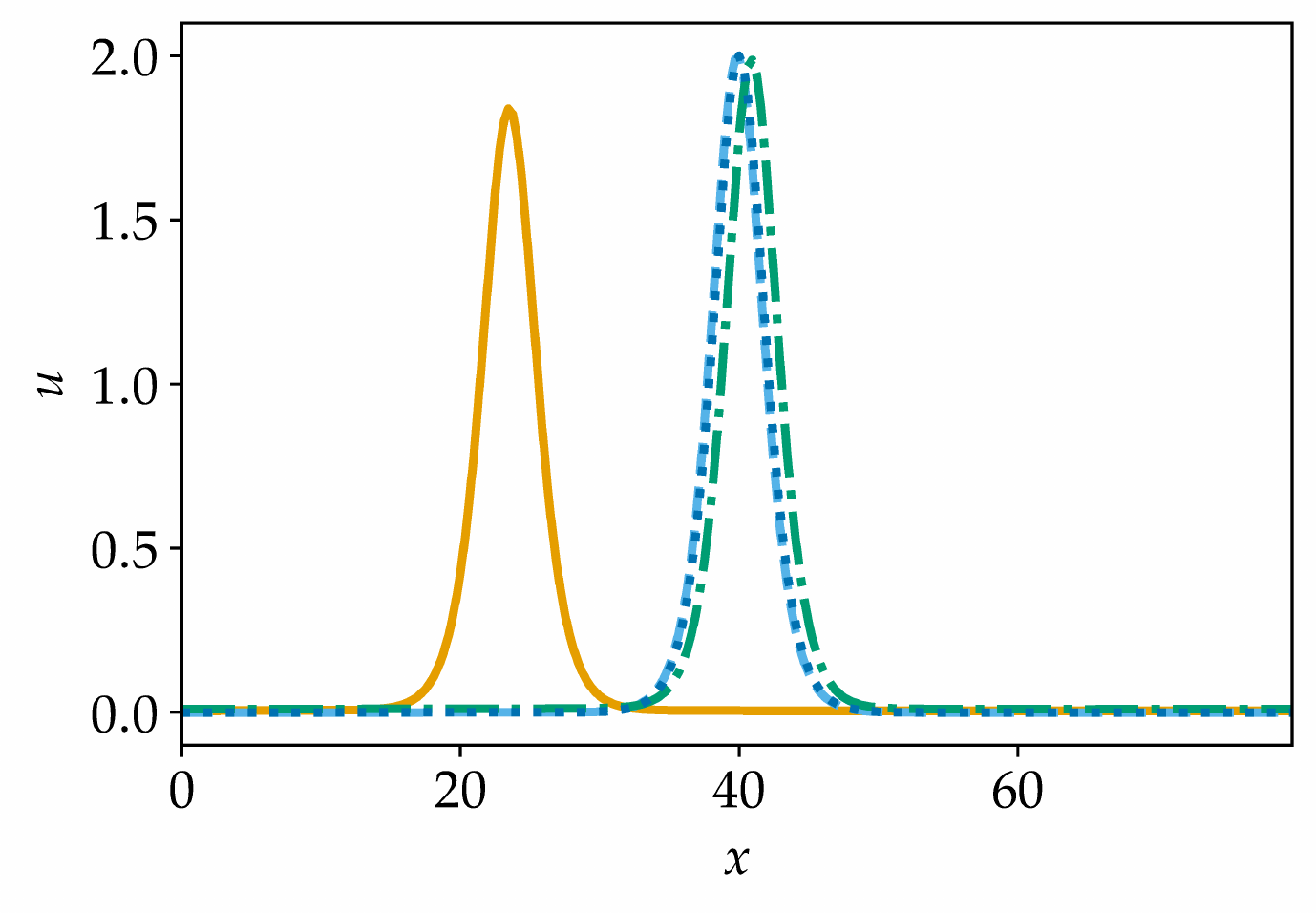}
    \caption{Analytical and numerical solutions at the final time
             with and without relaxation.}
    \label{fig:KdV-solution}
  \end{subfigure}%
  ~
  \begin{subfigure}[b]{0.49\textwidth}
    \centering
    \includegraphics[width=\textwidth]{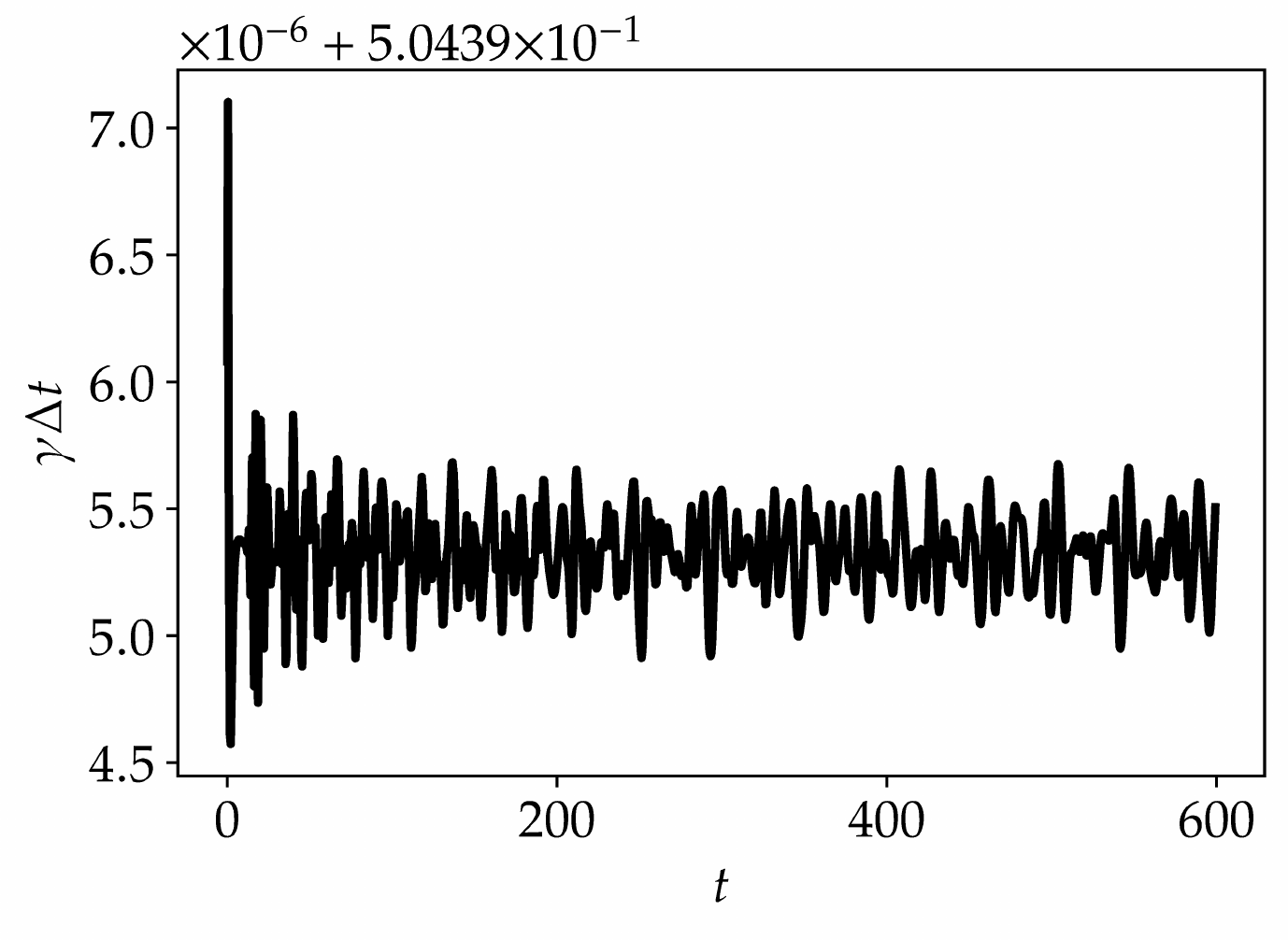}
    \caption{Scaled time step $\gamma \dt$ used by the relaxation method.}
    \label{fig:KdV-gamma-dt}
  \end{subfigure}%
  \caption{Results of numerical experiments with final time $t = 600$
           for the KdV equation \eqref{eq:KdV-semidiscrete}.}
  \label{fig:KdV-solution-and-gamma-dt}
\end{figure}

Based on observations for this problem, relaxation methods are particularly
interesting for implicit Runge-Kutta schemes. Firstly, the cost of solving
the scalar equation for the relaxation parameter $\gamma$ is negligible
compared to the cost of solving the stage equations. Secondly, the relaxation
approach can result in larger scaled time steps $\gamma \dt > \dt$. For the
scheme considered in this example, the baseline method uses time steps
$\dt = 0.5$ while the scaled timesteps $\gamma \dt$ of the relaxation method
were slightly larger with a median $\approx 0.504$, resulting in fewer
time steps needed to reach the final time, cf. Figure \ref{fig:KdV-gamma-dt}.

\subsection{Some Remarks on the Costs of Relaxation}
\label{sec:costs}

We have used SciPy \cite{virtanen2019scipy} to solve the scalar equations for
the relaxation parameter $\gamma$.  The implementations make
use of functions written in pure Python and are not adapted to the
specific problems.  A detailed assessment of the computational cost of
solving for $\gamma$ would require a more refined implementation
and is beyond the scope of this work.
Nevertheless, we give some preliminary discussion of the costs here,
emphasizing that these numbers should be viewed as
very loose upper bounds on the cost of relaxation.

If an explicit time-stepping method is used and the evaluation of the right-hand side is
inexpensive, the costs of naive implementations of the relaxation approach
are significant. For example, solving the Lotka-Volterra system with RK4
as discussed in Section~\ref{sec:qualitative-Lotka-Volterra} takes
\SI{5.14e-02}{s} without relaxation and \SI{1.18e-01}{s} with relaxation.
The phase space plot of the baseline scheme can be made to be visually
indistinguishable from the one of the relaxation method by reducing the
time step by a factor of four, resulting in a CPU time of \SI{1.97e-01}{s}.
Although the relaxation technique increases the runtime significantly in this
naive implemetation, it is still cheaper than decreasing the time step to get
basically the same results.

When implicit time-stepping is used, the cost of relaxation is less significant. For the
KdV equation discussed in Section~\ref{sec:KdV},
the baseline method needs \SI{6.36e+01}{s},
the relaxation method needs \SI{6.26e+01}{s},
and the projection method needs \SI{6.30e+01}{s}.
We see that the energy-conserving methods decrease the runtime despite the added
overhead of computing the relaxation or projection step. This can be explained
by noting that these methods yield improved accuracy and hence decreased costs to solve the nonlinear stage
equations. The relaxation method benefits additionally from taking larger
effective time steps since the relaxation parameter $\gamma > 1$.

\subsection{Superconvergence for Euclidean Hamiltonian Problems}
\label{sec:superconvergence}

The simplest Hamiltonian system is the harmonic oscillator
\eqref{eq:harmonic-osc}.
In Figure \ref{fig:harmonic-osc} we present convergence results for this
problem, integrated to final time $10$ and using four RK methods of orders
two to five. In the absence of relaxation, each scheme exhibits the expected
order of convergence. When relaxation is applied, the odd-order schemes
converge with a rate one order faster than their formal order.

\begin{figure}
\centering
  \begin{subfigure}{0.5\textwidth}
    \centering
    \includegraphics[width=\textwidth]{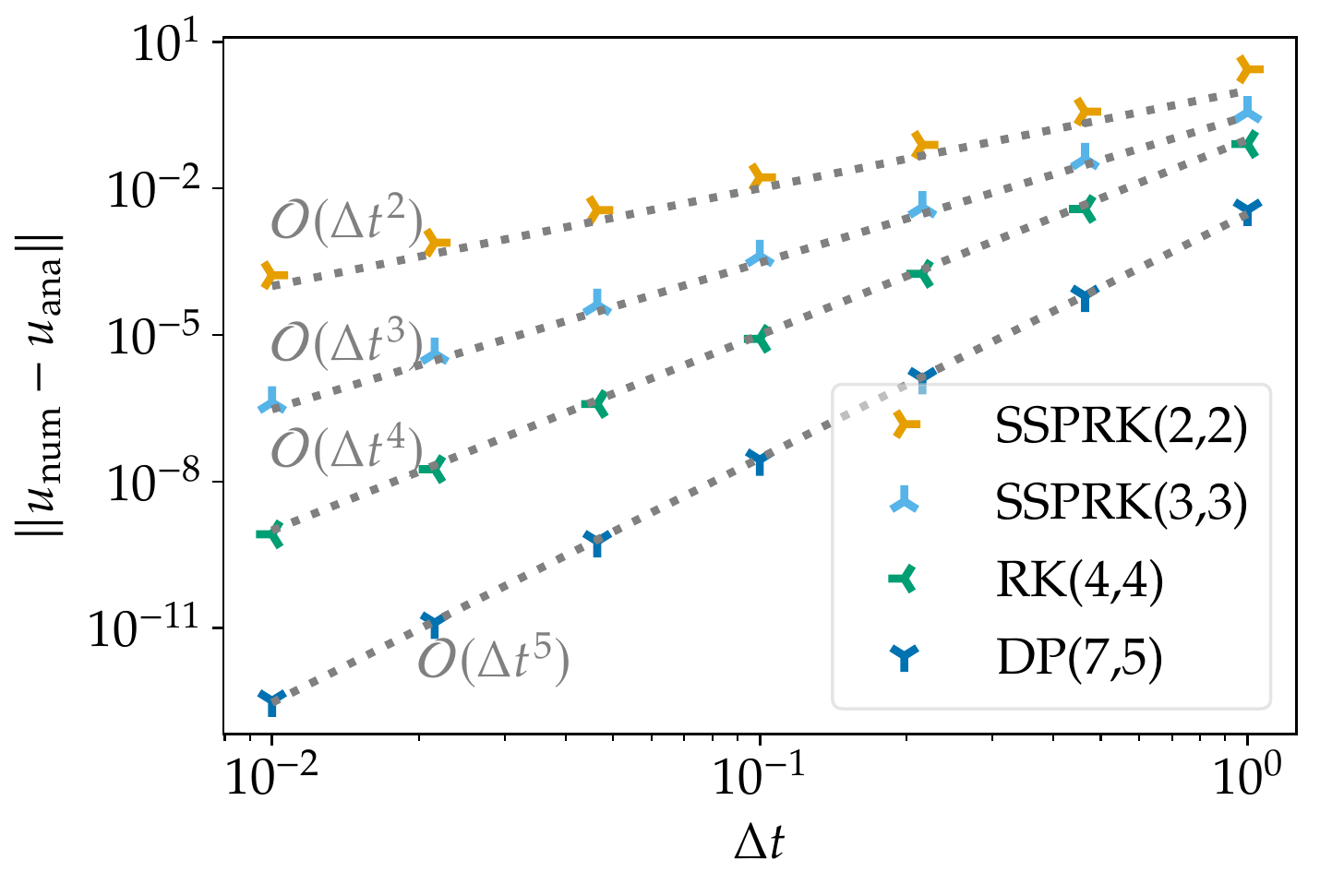}
    \caption{Baseline schemes.}
  \end{subfigure}%
  \begin{subfigure}{0.5\textwidth}
    \centering
    \includegraphics[width=\textwidth]{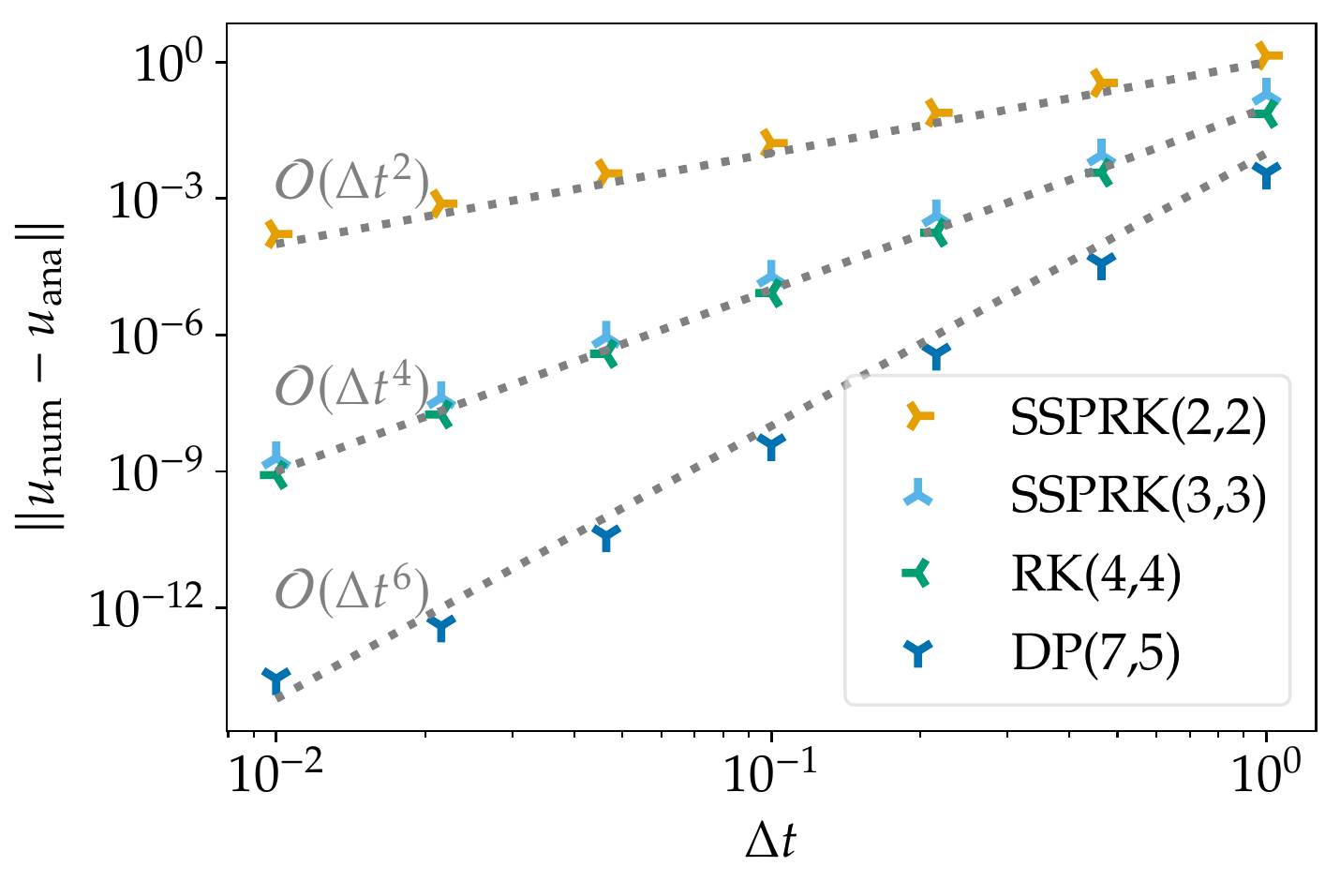}
    \caption{Relaxation schemes.}
  \end{subfigure}
  \caption{Convergence study using baseline Runge-Kutta and relaxation
            schemes for the harmonic oscillator \eqref{eq:harmonic-osc}.
            DP(7,5) is the fifth order method of
            Prince \& Dormand~\cite{prince1981high}.}
  \label{fig:harmonic-osc}
\end{figure}

Next, consider the nonlinear oscillator \eqref{eq:nonlinear-osc}.
In Figure \ref{fig:nonlinear-osc} we present convergence results
for this problem, integrated to final time $10$, again using RK
methods of orders two to five (but with different 3rd- and 5th-order
methods than in the previous example). Again,
the methods converge at the expected rate without relaxation, but the
odd-order methods show an increased rate of convergence when relaxation is applied.
Interestingly, the relaxation version of Heun's method has a significantly
smaller error than the classical RK(4,4) scheme, both with and without
relaxation.
\begin{figure}
\centering
  \begin{subfigure}{0.5\textwidth}
    \centering
    \includegraphics[width=\textwidth]{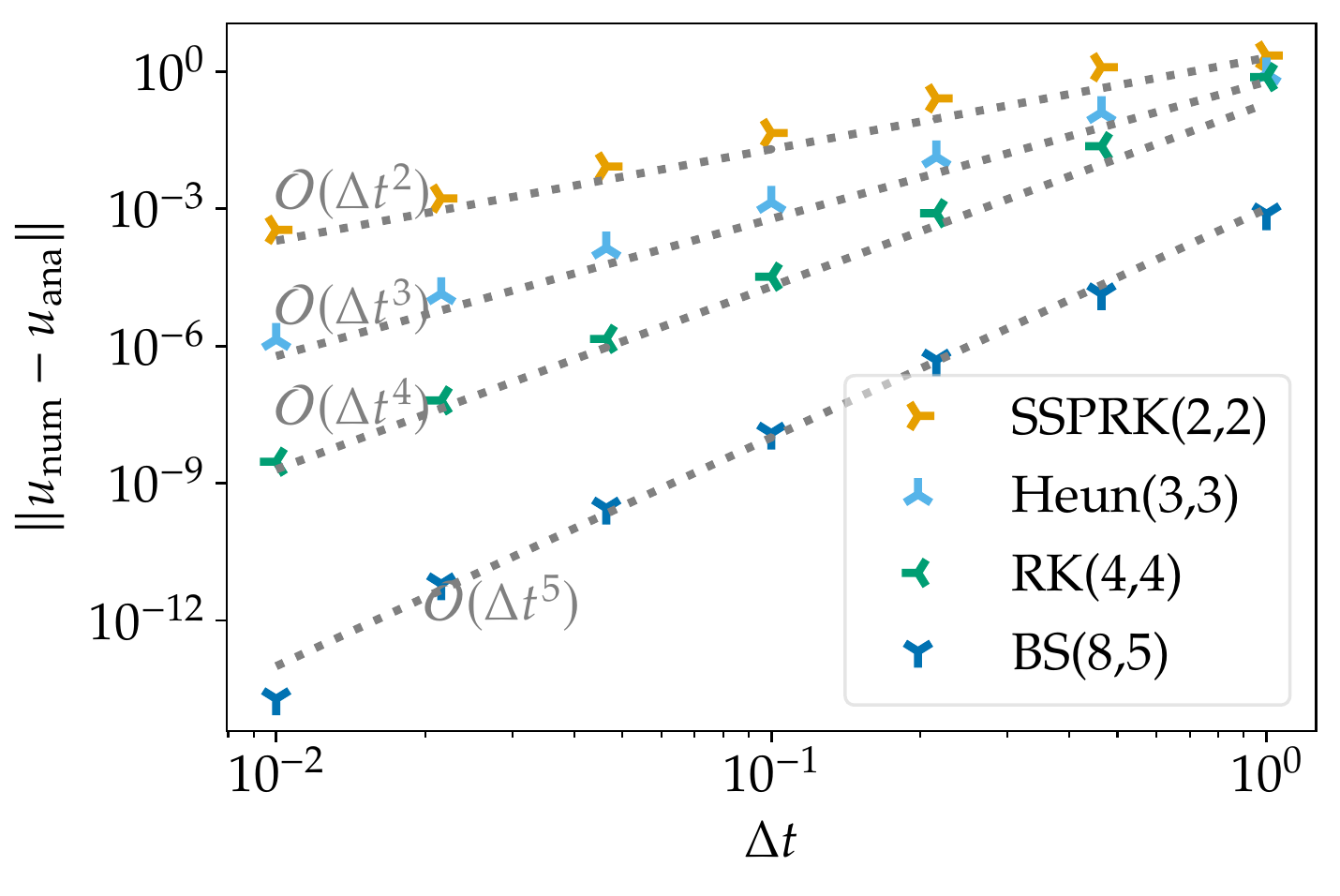}
    \caption{Baseline schemes.}
  \end{subfigure}%
  \begin{subfigure}{0.5\textwidth}
    \centering
    \includegraphics[width=\textwidth]{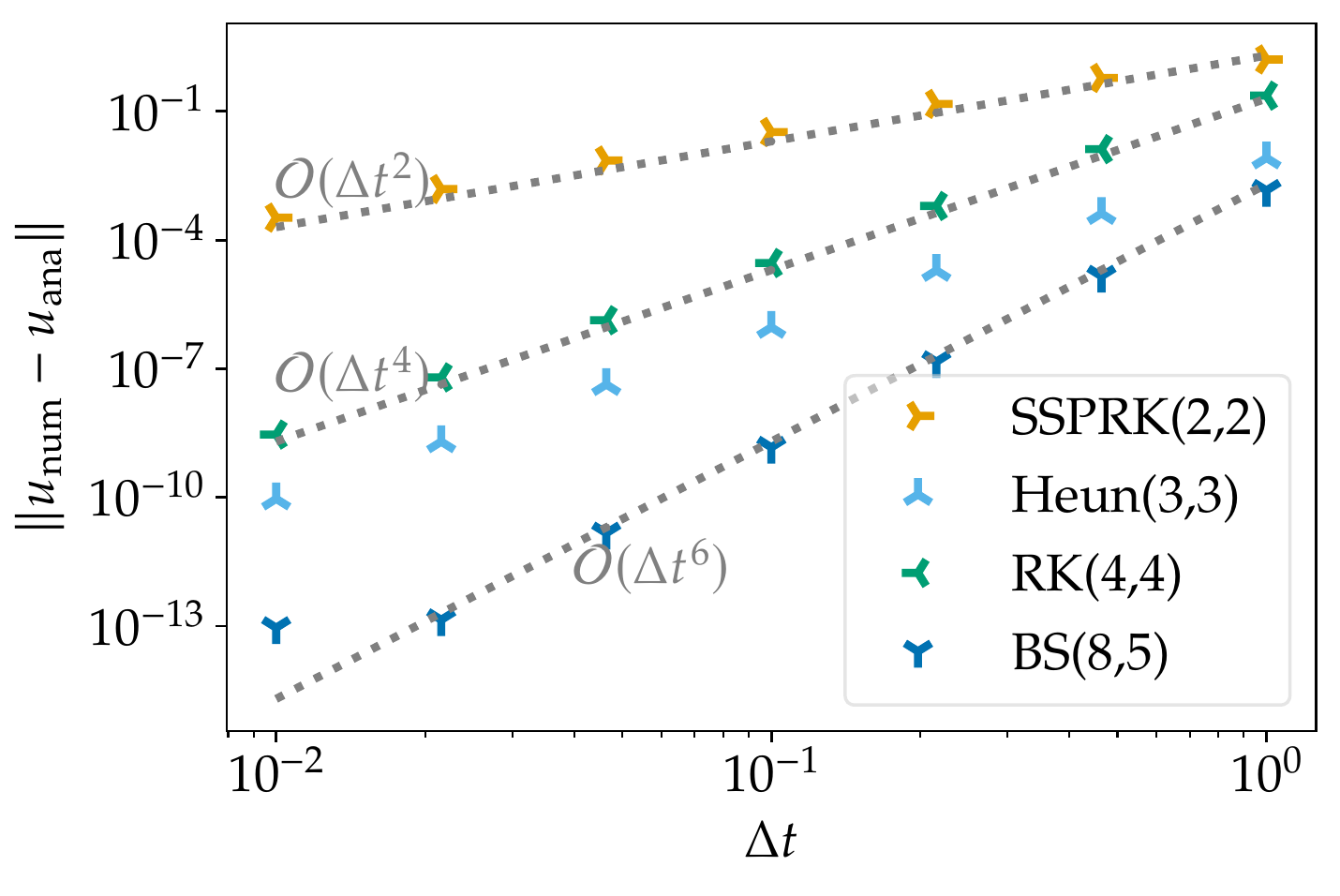}
    \caption{Relaxation schemes.}
  \end{subfigure}
  \caption{Convergence study using baseline Runge-Kutta and relaxation
           schemes for the nonlinear oscillator \eqref{eq:nonlinear-osc}.
           BS(8,5) is the fifth order method of
           Bogacki \& Shampine~\cite{bogacki1996efficient}.}
  \label{fig:nonlinear-osc}
\end{figure}

These two examples suggest a general result that we will state next.
Consider a Hamiltonian $H$ that is a smooth function of the squared
Euclidean norm, i.e.
\begin{equation}
  H(q,p) = G\bigl( (|q|^2 + |p|^2) / 2 \bigr),
\end{equation}
where $G$ is a smooth function. The corresponding Hamiltonian system is
\begin{equation}
\label{eq:nonlinear-Euclidean-Hamiltonian}
  \begin{cases}
    u'(t) = f(u(t)), \\
    u(0) = u_0,
  \end{cases}
  \quad
  u(t) = \begin{pmatrix} q(t) \\ p(t) \end{pmatrix},
  \quad
  f(q,p) = g\bigl( (|q|^2 + |p|^2) / 2 \bigr) \begin{pmatrix} p \\ -q \end{pmatrix},
\end{equation}
where $g = G'$.  We refer to \eqref{eq:nonlinear-Euclidean-Hamiltonian} as a
\emph{Euclidean Hamiltonian} problem.
For this class of problems, nominally odd order relaxation Runge-Kutta methods
are superconvergent.
\begin{theorem}
\label{thm:nonlinear-Euclidean-Hamiltonian}
  Consider the general nonlinear Euclidean Hamiltonian system
  \eqref{eq:nonlinear-Euclidean-Hamiltonian}
  and a Runge-Kutta method of order $p$ with $s$ stages.
  The corresponding RRK scheme conserving the Hamiltonian has an order
  of accuracy $p+1$ if $p$ is odd.
\end{theorem}

\begin{remark}
  Every skew-symmetric linear system in two space dimensions satisfies the
  Euclidean Hamiltonian structure above. Linear skew-symmetric systems
  in three space dimensions must have a zero eigenvalue and can consequently
  be reduced to a skew-symmetric system in two space dimensions.
\end{remark}

\begin{remark}
\label{rem:nonlinear-is-linear}
  Since $H$ is conserved, the analytical solution of any Euclidean Hamiltonian problem evolves
  identically to that of the corresponding linear system (with fixed value of $g(u) = g(u_0)$)
  \begin{equation}
    \begin{cases}
      u'(t) = f(u(t)), \\
      u(0) = u_0,
    \end{cases}
    \quad
    u(t) = \begin{pmatrix} q(t) \\ p(t) \end{pmatrix},
    \quad
    f(q,p) = g\bigl( (|q_0|^2 + |p_0|^2) / 2 \bigr) \begin{pmatrix} p \\ -q \end{pmatrix}.
  \end{equation}
\end{remark}

The proof of Theorem~\ref{thm:nonlinear-Euclidean-Hamiltonian} is presented in
Appendix \ref{sec:proofs}.
As can be seen there,
the Euclidean Hamiltonian structure is crucial. In particular, the
result does not hold for general skew-symmetric linear systems conserving
the Euclidean norm, e.g.
\begin{equation}
  u'(t) = L u(t),
  \quad
  u(0) = \begin{pmatrix} 1 \\ 0 \\ 0 \\ 0 \end{pmatrix},
  \quad
  L =
      \begin{pmatrix}
        0 & 0 & -1 & -1 \\
        0 & 0 & 0 & -1 \\
        1 & 0 & 0 & -1 \\
        1 & 1 & 1 & 0
      \end{pmatrix}.
\end{equation}
Furthermore, the superconvergence result does not hold for general linear
Hamiltonian systems where the Hamiltonian is not an affine function of the
squared Euclidean norm, e.g.
\begin{equation}
\begin{gathered}
  \od{}{t} \begin{pmatrix} q(t) \\ p(t) \end{pmatrix}
  =
  \begin{pmatrix}
    0 & \I \\
    -\I & 0
  \end{pmatrix}
  \nabla_{(q,p)} H(q,p),
  \quad
  H(q,p) = \frac{1}{2} q^T Q q + \frac{1}{2} p^T P p,
  \\
  Q =
  \begin{pmatrix}
    1 & 1 \\
    1 & 2
  \end{pmatrix},
  \quad
  P =
  \begin{pmatrix}
    3 & 2 \\
    2 & 4
  \end{pmatrix},
  \quad
  q(0) = \begin{pmatrix} 1 \\ 0 \end{pmatrix},
  \quad
  p(0) = \begin{pmatrix} 0 \\ 0 \end{pmatrix}.
\end{gathered}
\end{equation}
We have conducted further numerical experiments, supporting the hypothesis
that the superconvergence property of RRK schemes is restricted to
Euclidean Hamiltonian ones.

\section{\texorpdfstring{$N$}{N}-Body Applications}
\label{sec:n-body}

Energy preservation is particularly important in the solution of $N$-body
problem arising in astrophysics and molecular dynamics.
In this section we investigate the behavior of RRK methods on problems from
these two areas.

\subsection{Outer Solar System}
\label{sec:N-body-problems-solar-system}

Consider the outer solar system consisting of the sun (including the
mass of the inner planets), Jupiter, Saturn, Uranus, Neptune, and Pluto
This can be described by a Hamiltonian system with
\begin{align*}
    H(p,q) = \frac{1}{2} \sum_{i=0}^5 \frac{1}{m_i} p_i^T p_i
                - G \sum_{i=1}^5\sum_{j=0}^{i-1} \frac{m_i m_j}{\|q_i-q_j\|}.
\end{align*}
We take the masses and initial data exactly as given in
\cite[Section~I.2.3]{hairer2006geometric}.
Applying the second order accurate method SSPRK(2,2) with a (rather
large) time step size of $\dt = 200$ (days) for a total time of
\num{200000} (days) yields the results shown in
Figure~\ref{fig:outer-solar-system-baseline}.
While the orbits should be essentially stationary on this time scale, the
numerical approximation yields perturbed orbits with a drift of the radius, which
is most pronounced for the planets with the shortest orbits. Additionally, the Hamiltonian
(energy) grows considerably over time.

\begin{figure}[t]
\centering
  \begin{subfigure}[b]{0.49\textwidth}
    \centering
    \includegraphics[width=\textwidth]{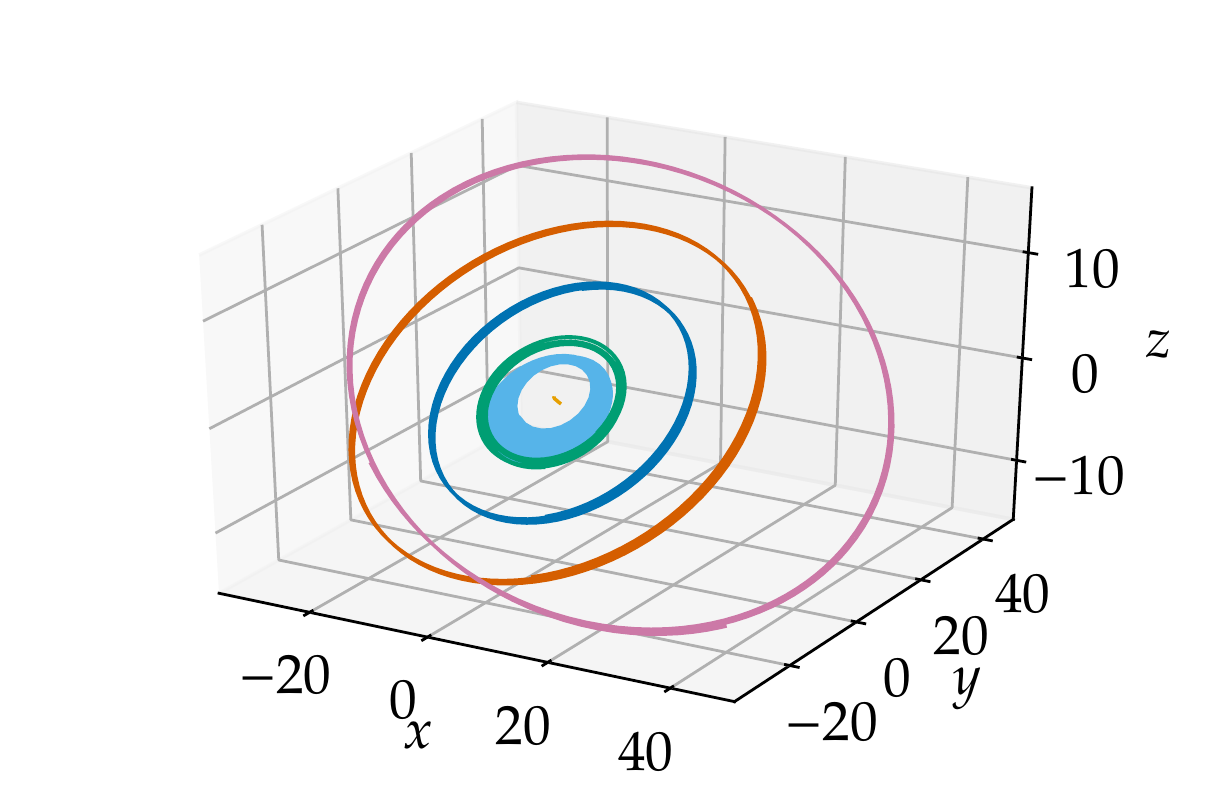}
    \caption{Orbits of the planets and the sun.}
  \end{subfigure}%
  ~
  \begin{subfigure}[b]{0.49\textwidth}
    \centering
    \includegraphics[width=\textwidth]{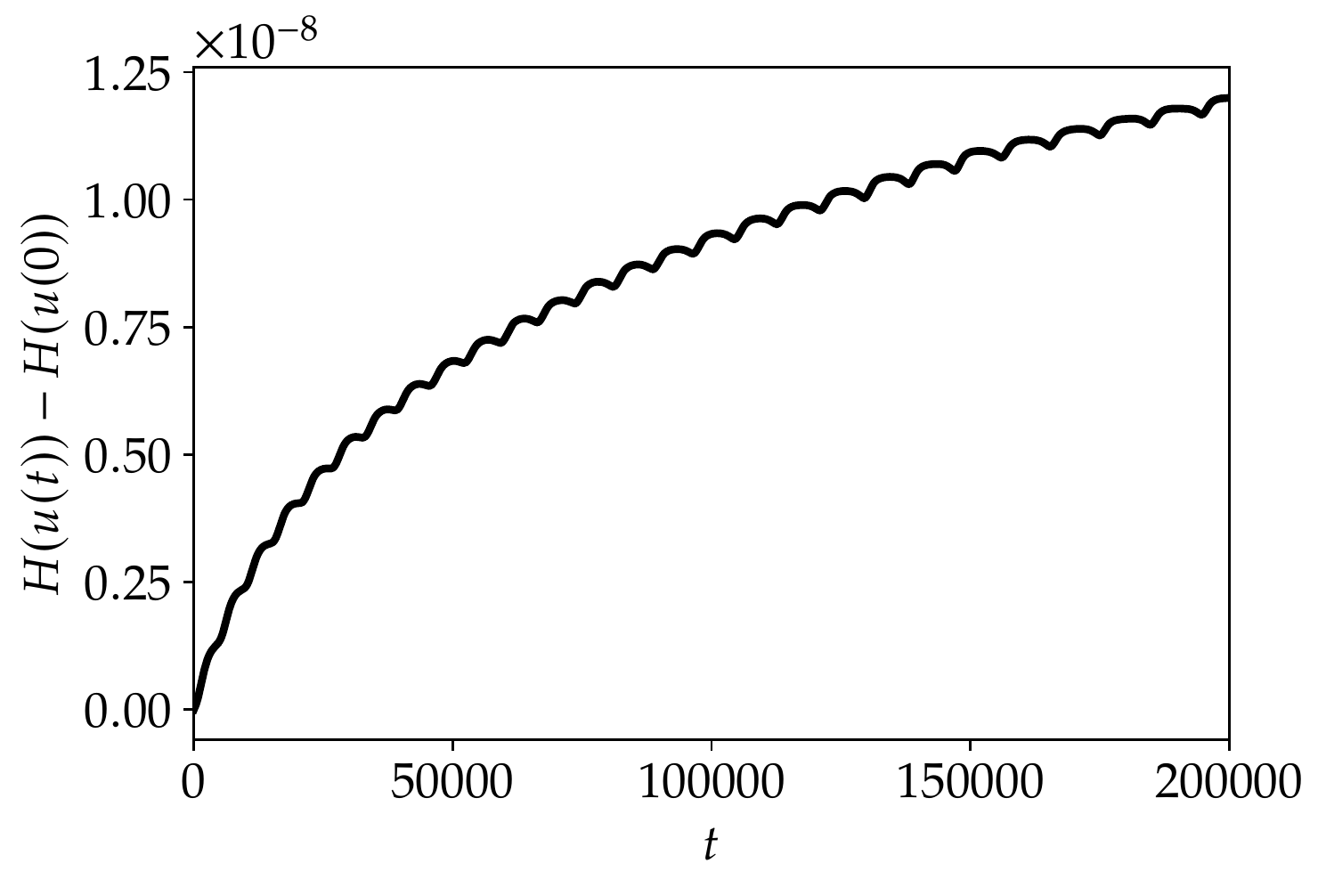}
    \caption{Change of the Hamiltonian over time.}
  \end{subfigure}%
  \caption{Simulation of the outer solar system using the baseline
           SSPRK(2,2) method.}
  \label{fig:outer-solar-system-baseline}
\end{figure}

Applying the relaxation approach to preserve the Hamiltonian during the
simulation does not necessarily increase the quality of the numerical
approximation in this case, as can be seen in
Figure~\ref{fig:outer-solar-system-relaxation}.
While the relaxation approach yields a constant (to machine accuracy)
Hamiltonian, the deviation in the orbits is similar or even worse.
This is in general agreement with the observations of
\cite[Example~IV.4.4]{hairer2006geometric}, where projection methods
conserving the Hamiltonian or even the Hamiltonian and the angular
momentum do not necessarily improve the quality of the numerical
solution. On the contrary, some orbits become even worse when applying
the projection methods used there. As stated in
\cite[Example~IV.4.4]{hairer2006geometric}: ``There is no doubt that
this problem contains a structure which cannot be correctly simulated
by methods that only preserve the total energy $H$ and the
angular momentum $L$''.

\begin{figure}
\centering
  \begin{subfigure}[b]{0.49\textwidth}
    \centering
    \includegraphics[width=\textwidth]{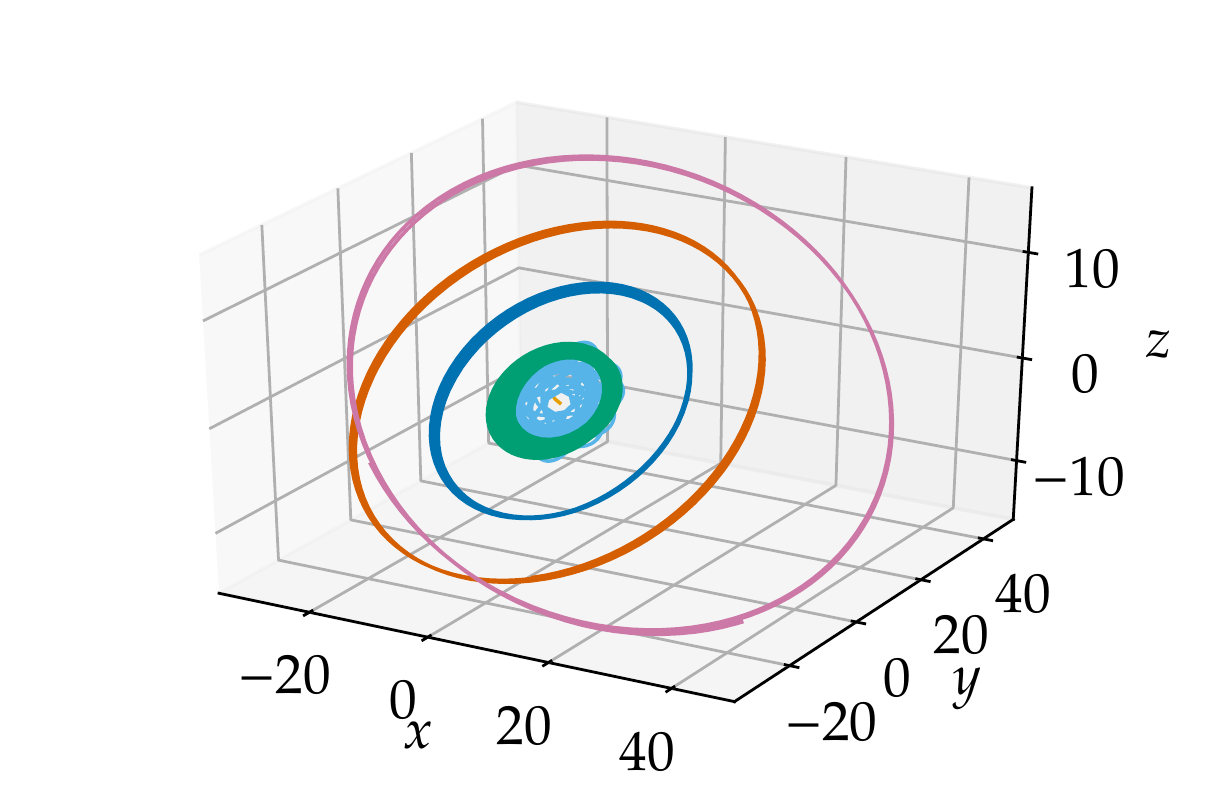}
    \caption{Orbits of the planets and the sun.}
  \end{subfigure}%
  ~
  \begin{subfigure}[b]{0.49\textwidth}
    \centering
    \includegraphics[width=\textwidth]{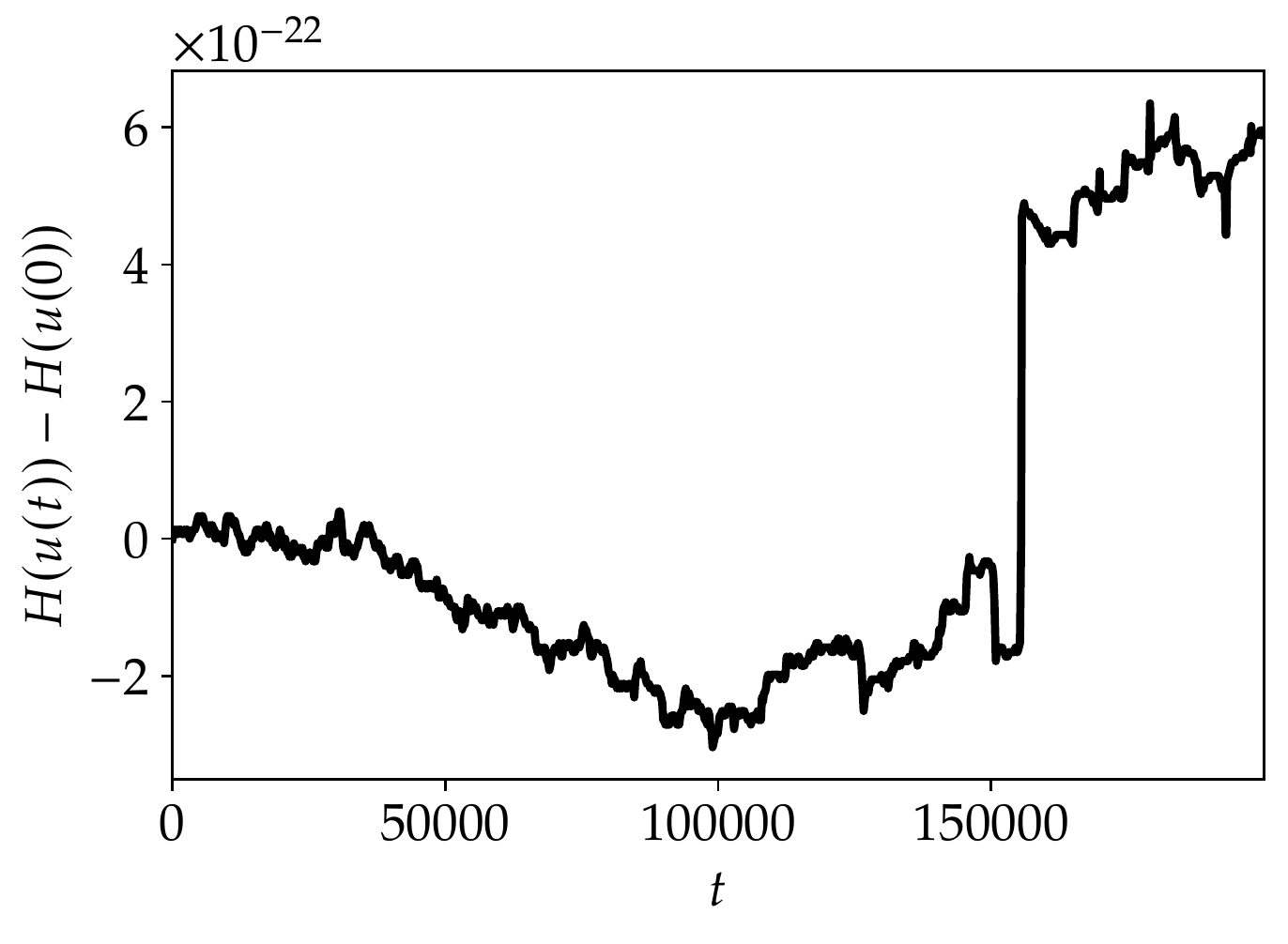}
    \caption{Change of the Hamiltonian over time.}
  \end{subfigure}%
  \caption{Simulation of the outer solar system using SSPRK(2,2)
           with relaxation to preserve the Hamiltonian (energy).}
  \label{fig:outer-solar-system-relaxation}
\end{figure}

\subsection{Molecular Dynamics}
\label{sec:N-body-problems-MD}

For the outer solar system discussed in the previous section, the
orbit of each body is important and has to be preserved as much as
possible. In contrast, in molecular dynamics simulations
``we do not need exact classical trajectories...
but must lay great emphasis on energy conservation as being of
primary importance.''~\cite[p.~99]{allen2017computer}
Thus it is reasonable to expect that energy conserving relaxation
Runge-Kutta methods can be applied advantageously to molecular dynamics
simulations.

Consider the example of a frozen argon crystal described by
Hairer, Lubich, and Wanner~\cite[Section~I.3.2]{hairer2006geometric},
consisting of argon atoms
interacting via the classical Lennard-Jones potential with an initial
condition close to the minimum energy state. As mentioned above, our
interest here is in statistical quantities such as the
total energy (Hamiltonian) and the temperature. While the Hamiltonian
is conserved, the temperature is expected to fluctuate around a certain
mean value. These properties are not preserved in general by numerical
discretizations.

Results of simulations of this system using the classical fourth
order Runge-Kutta method RK(4,4) with and without relaxation to
preserve the total energy are shown in Figure~\ref{fig:MD}.
As can be seen, the baseline scheme yields significant drift in the energy
and temperature.  Using instead the relaxation
approach, the Hamiltonian is conserved and the temperature fluctuates
around a mean value without significant drift.

To be competitive for molecular dynamics simulations, the cost of the
relaxation approach also has to be considered. While only a scalar
equation has to be solved per time step, it involves a computation
of the energy which is similarly expensive as computing the forces.
Hence, specialised implementations and considerations of using adapted
or inexact solutions for the relaxation parameter $\gamma$ should be
investigated, which is out of the scope of this article.

\begin{figure}
\centering
  \begin{subfigure}[b]{0.49\textwidth}
    \centering
    \includegraphics[width=\textwidth]{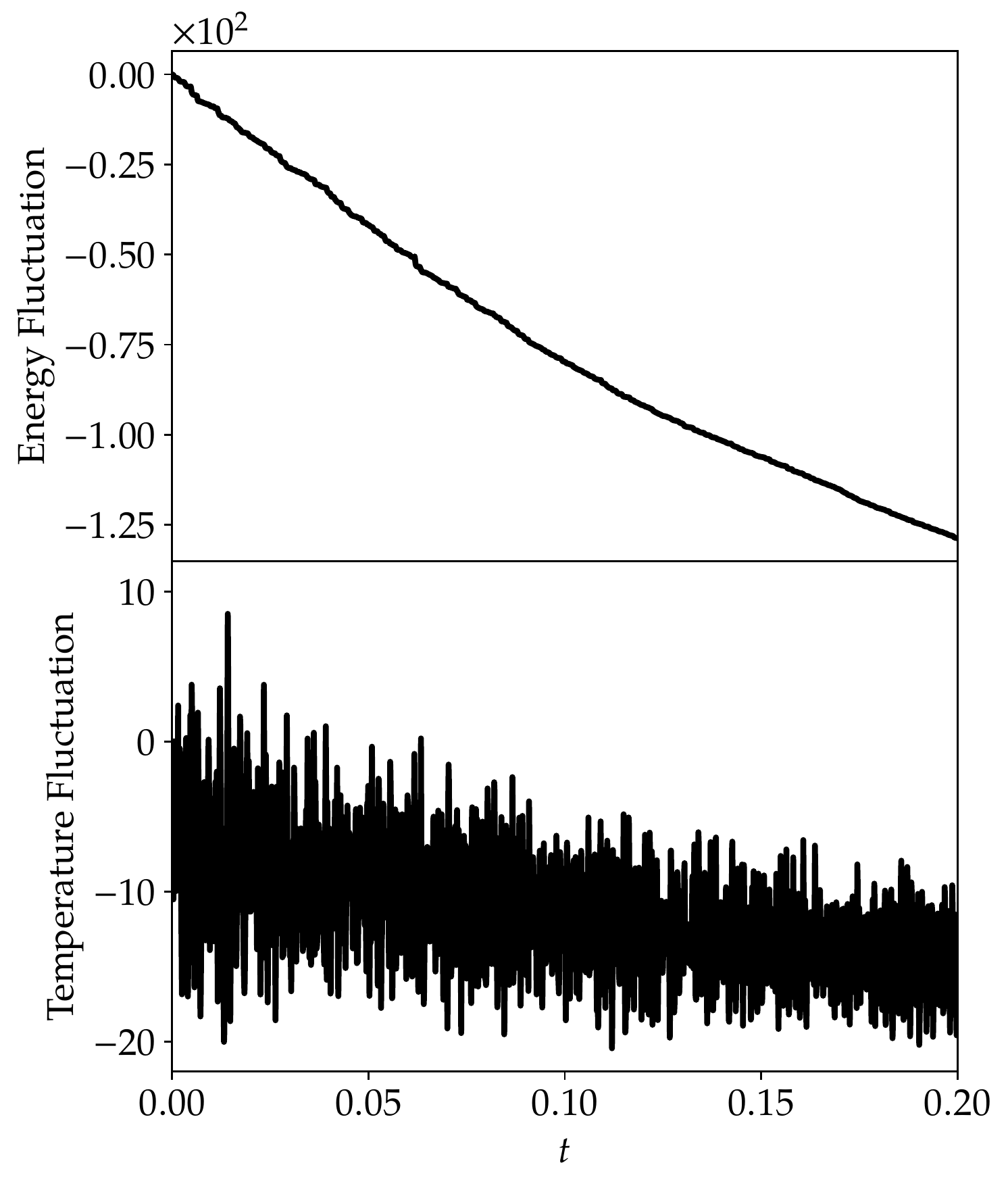}
    \caption{Baseline RK(4,4).}
  \end{subfigure}%
  ~
  \begin{subfigure}[b]{0.49\textwidth}
    \centering
    \includegraphics[width=\textwidth]{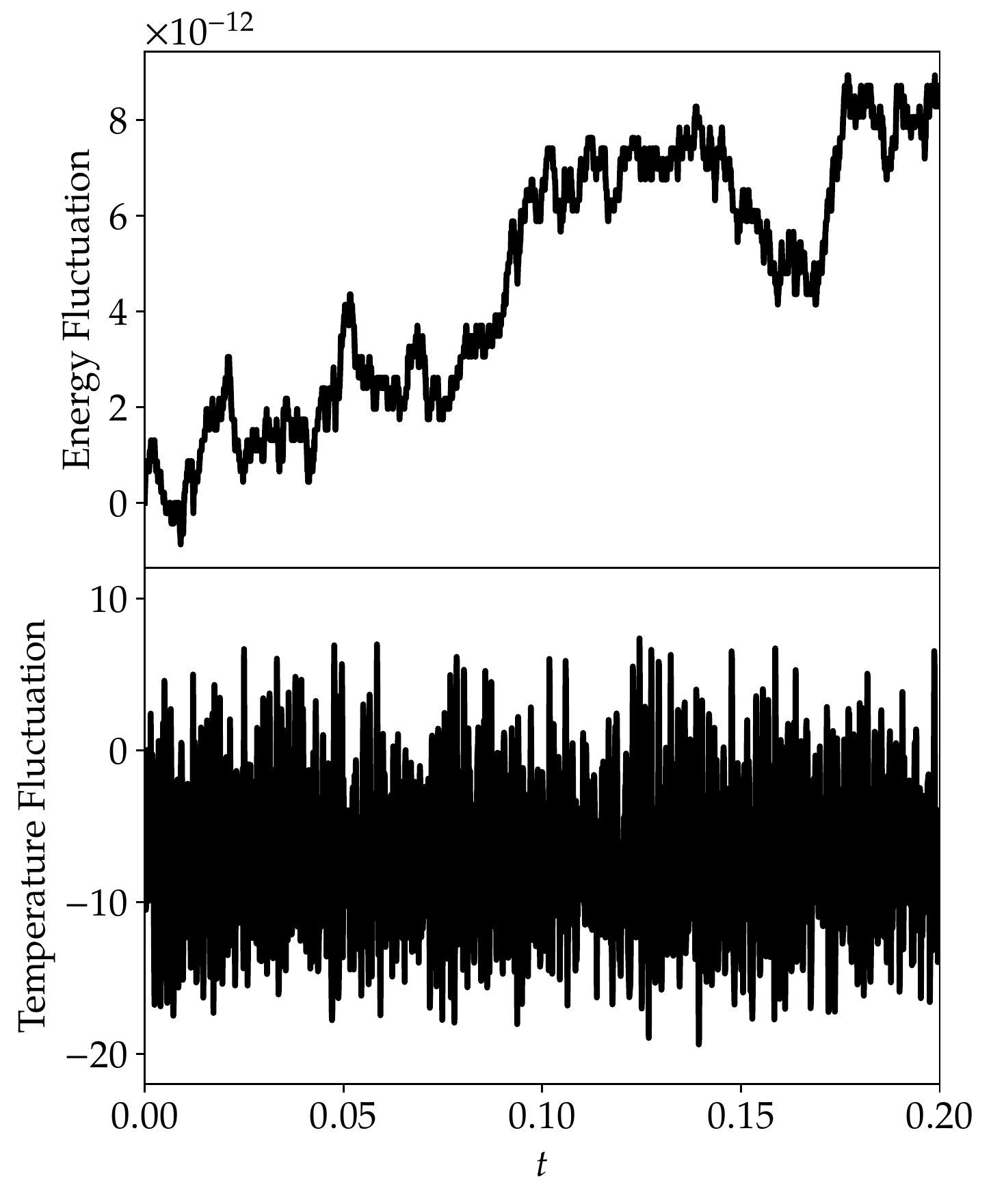}
    \caption{RK(4,4) with relaxation.}
  \end{subfigure}%
  \caption{Simulation of the frozen argon crystal with and without
           relaxation to preserve the Hamiltonian.}
  \label{fig:MD}
\end{figure}

\section{Conclusions}
\label{sec:conclusion}

We have studied relaxation Runge-Kutta (RRK) methods in the context of
differential equations with nonlinear invariants, especially Hamiltonian
systems. Similar to other structure preserving numerical schemes, RRK
methods result in an improved qualitative behavior of numerical solutions,
concerning e.g. the preservation of orbits or Poincaré sections
(Section~\ref{sec:qual}).
Additionally, they share advantageous quantitative properties with
symplectic methods such as a linear error growth for certain classes
of Hamiltonian problems (Section~\ref{sec:error-growth}). Furthermore, RRK
methods show a superconvergence property for Euclidean Hamiltonian systems
(Section~\ref{sec:superconvergence}). Finally, our limited results suggest
that RRK methods may be useful for Hamiltonian problems, such as those in
molecular dynamics, where statistical averages are of interest (Section~\ref{sec:n-body}).

For some classes of Hamiltonian problems, explicit symplectic methods such
as the Störmer Verlet scheme are available. Compared to these methods,
RRK methods have to solve a single scalar equation per time step which can
be a certain bottleneck. To be fully competetive, specialised implementations
will studied in a future work. On the other hand, RRK methods can be of
arbitrary order and allow more freedom in their design.

Whenever implicit schemes are advantageous, e.g. because of stiffness
restrictions, the relaxation approach provides an inexpensive means to
conserve important nonlinear quantities, which is usually
only possible for fully implicit methods. Hence, diagonally implicit RRK methods
seem very competitive for such problems.
In future work, we plan to deepen the investigation of RRK methods in this
context and study the combination of the relaxation approach with several
classes of schemes such as implicit-explicit (IMEX) methods and splitting
algorithms.

Another direction of future research is the application of relaxation
to enforce conservation with a Hamiltonian PDE semi-discretization that
is not conservative.  In this context the relaxation process would be used
to remove the conservation error created by the spatial discretization.
Some numerical experiments (not included here)
suggest that the relaxation approach can be applied successfully in this
context as long as the discretization errors in space and time are of a similar
order of magnitude. Nevertheless, deeper theoretical and empirical studies
are necessary.

\appendix
\section*{Acknowledgements}

Research reported in this publication was supported by the
King Abdullah University of Science and Technology (KAUST).
The authours would like to thank Ernst Hairer for a discussion of
symplecticity and the preservation of phase space volume.

\section{Superconvergence Theorem Proofs}
\label{sec:proofs}

To facilitate the proof of the general result \ref{thm:nonlinear-Euclidean-Hamiltonian},
we first present and prove a version restricted to linear problems.
\begin{theorem}
\label{thm:linear-Euclidean-Hamiltonian}
  Consider the linear Euclidean Hamiltonian system
    \begin{equation}
    \label{eq:linear-Euclidean-Hamiltonian}
      \begin{cases}
        u'(t) = L u(t), \\
        u(0) = u_0,
      \end{cases}
      \quad
      u(t) = \begin{pmatrix} q(t) \\ p(t) \end{pmatrix},
      \quad
      L = \alpha \begin{pmatrix} 0 & \I \\ -\I & 0 \end{pmatrix}.
    \end{equation}
  and an explicit Runge-Kutta method of order $p \in \N$ with $s \geq p$
  stages.
  The corresponding RRK scheme conserving the Hamiltonian has an order
  of accuracy $p+1$ if $p$ is odd.
\end{theorem}

In the proof of Theorem~\ref{thm:linear-Euclidean-Hamiltonian}, the following
lemma will be used.
\begin{lemma}
\label{lem:combinatorial-lemma}
  For $s,m \in \N$, $2m \le s+1$,
  \begin{equation}
    \sum_{n=\max\set{1-m, m-s}}^{\min\set{m-1, s-m}} (-1)^n \frac{1}{(m-n)! (m+n)!}
    =
    -2 (-1)^m \frac{1}{(2m)!}.
  \end{equation}
\end{lemma}
\begin{proof}
  Use some explicit calculations or the Mathematica \cite{mathematica12}
  notebook \texttt{Combinatorial\_Lemma.nb} in the repository
  \cite{ranocha2020relaxationHamiltonianRepro}.
\end{proof}

\begin{proof}[Proof of Theorem~\ref{thm:linear-Euclidean-Hamiltonian}]
  It suffices to consider the first step. The baseline RK method
  starting from $u_0$ yields
  \begin{equation}
    u_+ = u_0 + \sum_{k=1}^s \alpha_k \dt^k L^k u_0,
  \end{equation}
  where $\alpha_k$ are the monomial coefficients of the corresponding
  stability polynomial. The relaxation method results in the new value
  \begin{equation}
    u_+^\gamma = u_0 + \gamma \sum_{k=1}^s \alpha_k \dt^k L^k u_0
  \end{equation}
  with squared Euclidean norm
  \begin{equation}
    \norm{u_+^\gamma}^2
    =
    \left(
      2 \gamma \sum_{l=1}^{\floor{s/2}} (-1)^l \alpha_{2l} \dt^{2l} \norm{L^l u_0}^2
      + \gamma^2 \sum_{\mathclap{\substack{j,k=1 \\ j+k\text{ even}}}}^s (-1)^{(j-k)/2} \alpha_j \alpha_k h^{j+k} \norm{L^{(j+k)/2} u_0}^2
    \right) \norm{u_0}^2.
  \end{equation}
  The second sum can also be written as
  \begin{equation}
    \sum_{\mathclap{\substack{j,k=1 \\ j+k\text{ even}}}}^s \alpha_j \alpha_k (-1)^{(j-k)/2} h^{j+k} \norm{L^{(j+k)/2} u_0}^2
    =
    \sum_{m=1}^s \sum_{n=\max\set{1-m, m-s}}^{\min\set{m-1, s-m}} (-1)^n \alpha_{m-n} \alpha_{m+n} \norm{L^m u_0}^2.
  \end{equation}
  The non-zero value of the relaxation parameter conserving the
  Euclidean norm is
  \begin{equation}
  \label{eq:linear-superconvergence-gamma-general}
    \gamma
    =
    \frac{
      - 2 \sum_{l=1}^{\floor{s/2}} (-1)^l \alpha_{2l} \dt^{2l} \norm{L^l u_0}^2
    }{
      \sum_{m=1}^s \sum_{n=\max\set{1-m, m-s}}^{\min\set{m-1, s-m}} (-1)^n \alpha_{m-n} \alpha_{m+n} \norm{L^m u_0}^2
    }.
  \end{equation}

  For a $p$-th order baseline scheme,
  \begin{equation}
    \alpha_k = \frac{1}{k!}, \qquad k \in \set{1, \dots, p}.
  \end{equation}
  Hence, for odd $p$, the numerator of $\gamma$
  \eqref{eq:linear-superconvergence-gamma-general} is
  \begin{multline}
    - 2 \sum_{l=1}^{\floor{s/2}} (-1)^l \alpha_{2l} \dt^{2l} \norm{L^l u_0}^2
    \\
    =
    - 2 \sum_{l=1}^{\floor{p/2}} (-1)^l \frac{1}{(2l)!} \dt^{2l} \norm{L^l u_0}^2
    - 2 (-1)^{(p+1)/2} \alpha_{p+1} \dt^{p+1} \norm{L^{(p+1)/2} u_0}^2
    + \O(\dt^{p+3}).
  \end{multline}
  Because of Lemma~\ref{lem:combinatorial-lemma}, the denominator of $\gamma$
  \eqref{eq:linear-superconvergence-gamma-general} for odd $p$ is
  \begin{multline}
    \sum_{m=1}^s \sum_{n=\max\set{1-m, m-s}}^{\min\set{m-1, s-m}} (-1)^n \alpha_{m-n} \alpha_{m+n} \norm{L^m u_0}^2
    \\
    =
    - 2 \sum_{m=1}^{\floor{p/2}} (-1)^m \frac{1}{(2m)!} \norm{L^m u_0}^2
    - 2 (-1)^{(p+1)/2} \frac{1}{(p+1)!} \dt^{p+1} \norm{L^{(p+1)/2} u_0}^2
    + \O(\dt^{p+3}).
  \end{multline}
  Thus,
  \begin{equation}
  \label{eq:linear-superconvergence-gamma-special}
    \gamma
    =
    1 - 2 (-1)^{(p+1)/2} \left( \alpha_{p+1} - \frac{1}{(p+1)!} \right)
    \dt^{p-1} \frac{\norm{L^{(p+1)/2} u_0}^2}{\norm{L u_0}^2}
    + \O(\dt^{p+1}).
  \end{equation}

  Comparing the analytical solution
  \begin{equation}
    u(\gamma \dt)
    =
    \left( \sum_{k=0}^\infty \frac{\gamma^k \dt^k}{k!} L^k \right) u_0
  \end{equation}
  with the RRK solution,
  \begin{equation}
    u_+^\gamma - u(\gamma \dt)
    =
    \left(
      \sum_{k=1}^p \frac{1}{k!} (\gamma - \gamma^k) \dt^k L^k
    \right) u_0
    + \left( \alpha_{p+1} \gamma - \frac{\gamma^{p+1}}{(p+1)!} \right) \dt^{p+1} L^{p+1} u_0
    + \O(\dt^{p+2}).
  \end{equation}
  Inserting $\gamma$ \eqref{eq:linear-superconvergence-gamma-special} and
  expanding the term $(\gamma - \gamma^2) \dt^2 L^2 u_0$ results in
  \begin{equation}
    u_+^\gamma - u(\gamma \dt)
    =
    \left( \alpha_{p+1} - \frac{1}{(p+1)!} \right)
    \left(
      (-1)^{(p+1)/2} \frac{\norm{L^{(p+1)/2} u_0}^2}{\norm{L u_0}^2} L^2 u_0
      + L^{p+1} u_0
    \right)
    \dt^{p+1}
    + \O(\dt^{p+2}).
  \end{equation}
  Finally, the second term in brackets vanishes because of the special
  structure of $L$ \eqref{eq:linear-Euclidean-Hamiltonian}.
\end{proof}

To prove Theorem~\ref{thm:nonlinear-Euclidean-Hamiltonian}, expansions using
rooted trees will be applied, cf. \cite[Chapter~3]{butcher2016numerical}.
The following structural results will be used.
\begin{lemma}
\label{lem:nonlinear-Euclidean-Hamiltonian-terms-1}
  For the Euclidean Hamiltonian system \eqref{eq:nonlinear-Euclidean-Hamiltonian},
  $m \in \N$, and $n \in \set{0, \dots, m}$, there exists a smooth function
  $h_{m,n}$ such that
  \begin{equation}
  \label{eq:nonlinear-Euclidean-Hamiltonian-terms-1}
    f^{(m)}(\underbrace{f, \dots, f}_{n \text{ terms}}, \underbrace{u, \dots, u}_{m-n \text{ terms}})
    =
    \begin{cases}
      h_{m,n}(\norm{u}^2/2) f(u), & \text{if } n \text{ is even}, \\
      h_{m,n}(\norm{u}^2/2) u, & \text{if } n \text{ is odd}.
    \end{cases}
  \end{equation}
\end{lemma}
\begin{proof}[Proof by induction]
  The induction hypothesis is fulfilled for the basic cases
  $m = 1$ \& $n \in \set{0,1}$, since
  \begin{equation}
  \label{eq:nonlinear-Euclidean-Hamiltonian-f'f}
  \begin{gathered}
    f'(q, p)
    =
    \begin{pmatrix}
      g' p q^T & g' p p^T + g \I \\
      - g' q q^T - g \I & - g' q p^T
    \end{pmatrix},
    \\
    f' f
    =
    - g^2 \begin{pmatrix} q \\ p \end{pmatrix}
    \propto
    u,
    \quad
    f' u
    =
    \bigl( g' (|q|^2 + |p|^2) + g \bigr) \begin{pmatrix} p \\ -q \end{pmatrix}
    \propto
    f,
  \end{gathered}
  \end{equation}
  and $g, g'$ depend on $\norm{u}^2$.

  The induction step from $m$ to $m+1$ can be carried out by differentiating
  the identity \eqref{eq:nonlinear-Euclidean-Hamiltonian-terms-1}.
  For even $n$, the derivative with respect to $u^j$ of the right hand side is
  \begin{equation}
    \partial_j h_{m,n} f^i
    =
    h_{m,n}' u^j f^i + h_{m,n} f^i_j.
  \end{equation}
  Multiplication by $u^j$ and $f^j$, respectively, yield
  \begin{equation}
  \begin{aligned}
    (\partial_j h_{m,n} f^i) u^j
    &=
    (h_{m,n}' \norm{u}^2 + g' \norm{u}^2 + g) f^i,
    \\
    (\partial_j h_{m,n} f^i) f^j
    &=
    -g^2 h_{m,n} u^i.
  \end{aligned}
  \end{equation}
  Similarly, for odd $n$, the derivative of the right hand side is
  \begin{equation}
    \partial_j h_{m,n} u^i
    =
    h_{m,n}' u^j u^i + h_{m,n} \delta^i_j
  \end{equation}
  and multiplication by $u^j$ and $f^j$, respectively, result in
  \begin{equation}
  \begin{aligned}
    (\partial_j h_{m,n} u^i) u^j
    &=
    (h_{m,n}' \norm{u}^2 + h_{m,n}) u^i,
    \\
    (\partial_j h_{m,n} u^i) f^j
    &=
    h_{m,n} f^i.
  \end{aligned}
  \end{equation}
  For all $n$, multiplication of the derivative of the right hand side by
  $u^j$ doesn't change the direction while multiplication by $f^j$ flips the
  direction from $u$ to $f$ and vice versa.

  The derivative with respect to $u^j$ of the left hand side of
  \eqref{eq:nonlinear-Euclidean-Hamiltonian-terms-1} is
  \begin{equation}
  \begin{aligned}
    &\quad
    \partial_j f^i_{j_1 \dots j_m} f^{j_{1}} \dots f^{j_{n}} u^{j_{n+1}} \dots u^{j_{m}}
    \\
    &=
    f^i_{j_1 \dots j_m j} f^{j_{1}} \dots f^{j_{n}} u^{j_{n+1}} \dots u^{j_{m}}
    \\&\quad
    +
    f^i_{j_1 \dots j_m} f^{j_{1}}_j f^{j_{2}} \dots f^{j_{n}} u^{j_{n+1}} \dots u^{j_{m}}
    + \dots +
    f^i_{j_1 \dots j_m} f^{j_{1}} \dots f^{j_{n-1}} f^{j_{n}}_j u^{j_{n+1}} \dots u^{j_{m}}
    \\&\quad
    +
    f^i_{j_1 \dots j_m} f^{j_{1}} \dots f^{j_{n}} \delta^{j_{n+1}}_j u^{j_{n+2}} \dots u^{j_{m}}
    + \dots +
    f^i_{j_1 \dots j_m} f^{j_{1}} \dots f^{j_{n}} u^{j_{n+1}} \dots u^{j_{m-1}} \delta^{j_{m}}_j.
  \end{aligned}
  \end{equation}
  Multiplication of all but the first term on the right hand side by $u^j$
  doesn't change the direction while multiplication by $f^j$ flips the
  direction from $u$ to $f$ and vice versa, exactly as for the derivative
  of the right hand side of \eqref{eq:nonlinear-Euclidean-Hamiltonian-terms-1}.
  Hence, the term
  $f^i_{j_1 \dots j_m j} f^{j_{1}} \dots f^{j_{n}} u^{j_{n+1}} \dots u^{j_{m}}$
  must show the same behavior, proving the induction hypothesis
  \eqref{eq:nonlinear-Euclidean-Hamiltonian-terms-1} for $m+1$ instead of $m$.
\end{proof}

\begin{lemma}
\label{lem:nonlinear-Euclidean-Hamiltonian-terms}
  For the Euclidean Hamiltonian system \eqref{eq:nonlinear-Euclidean-Hamiltonian}
  and a rooted tree $t$,
  \begin{equation}
  \label{eq:nonlinear-Euclidean-Hamiltonian-terms}
    \begin{cases}
      F(t)(u_0) \parallel f(u_0), & \text{if } |t| \text{ is odd}, \\
      F(t)(u_0) \parallel u_0,    & \text{if } |t| \text{ is even},
    \end{cases}
  \end{equation}
  where $\parallel$ indicates that two vectors are parallel.
\end{lemma}
\begin{proof}
  The result is proved by induction using the hypothesis
  ``Consider a rooted tree $t = [t_1, \dots, t_m]$ with elementary differential
  $F(t)(u_0) = f^{(m)}(u_0)\bigl( F(t_1)(u_0), \dots, F(t_m)(u_0) \bigr)$.
  If a leaf is added to one of the $t_i$ or the direction of one of the arguments
  of $f^{(m)}(u_0)$ is changed from $u_0$ to $f(u_0)$ or vice versa, the direction
  of $F(t)(u_0)$ changes from $u_0$ to $f(u_0)$ and vice versa.
  Additionally, \eqref{eq:nonlinear-Euclidean-Hamiltonian-terms} holds.''

  The induction hypothesis is fulfilled for the base case $|t| = 1$ because of
  \eqref{eq:nonlinear-Euclidean-Hamiltonian-f'f}.

  Induction step: Appending a leaf or changing the direction of one of the
  arguments flips the direction because of the induction hypothesis.
  The bushy trees behave as desired because of
  Lemma~\ref{lem:nonlinear-Euclidean-Hamiltonian-terms-1}.
\end{proof}

\begin{proof}[Proof of Theorem~\ref{thm:nonlinear-Euclidean-Hamiltonian}]
  To generalize the approach for the linear case in the proof of
  Theorem~\ref{thm:linear-Euclidean-Hamiltonian}, the leading order
  term of $\gamma - 1$ has to be computed and $u_+^\gamma$ has to be
  compared with $u(\gamma \dt)$.

  The relaxation parameter $\gamma$ can be written as
  \cite[eq. (11) and Remark~4]{ketcheson2019relaxation}
  \begin{equation}
    \gamma
    =
    \frac{
      2 \sum_{i,j=1}^s b_i a_{ij} \scp{f_i}{f_j}
    }{
      \sum_{i,j=1}^s b_i b_j \scp{f_i}{f_j}
    }
    =
    \frac{
      2 \norm{u_0}^2 - 2 \scp{u_+}{u_0}
    }{
      \norm{u_+ - u_0}^2
    }
    =
    \frac{
      2 \norm{u_0}^2 - 2 \scp{u_+}{u_0}
    }{
      2 \norm{u_0}^2 - 2 \scp{u_+}{u_0} + (\norm{u_+}^2 - \norm{u_0}^2).
    }.
  \end{equation}
  Hence, the denominator of $\gamma$ is the numerator plus a high order
  correction $\norm{u_+}^2 - \norm{u_0}^2$, exactly as in the linear case.

  For $n \in \N$, the approximate solution of the baseline RK scheme after
  one step can be expanded as \cite[eq.~(313c)]{butcher2016numerical}
  \begin{equation}
  \label{eq:expansion-uplus}
    u_+
    =
    u_0 + \sum_{|t| \leq n} \frac{1}{\sigma(t)} \Phi(t) \dt^{|t|} F(t)(u_0)
    + \O(\dt^{n+1}).
  \end{equation}
  For an at least second order baseline scheme, the numerator of $\gamma$ is
  \begin{multline}
    2 \norm{u_0}^2 - 2 \scp{u_+}{u_0}
    =
    - 2 \sum_{|t| = 2} \frac{1}{\sigma(t)} \Phi(t) \dt^{|t|} F(t)(u_0)
    + \O(\dt^{3})
    \\
    =
    - 2 \underbrace{\sum_{i=1}^s b_i c_i}_{= 1/2} \dt^2 \scp{ f' f(u_0) }{ u_0 }
    + \O(\dt^{3}).
  \end{multline}
  Because of \eqref{eq:nonlinear-Euclidean-Hamiltonian-f'f},
  \begin{equation}
    \scp{ f' f(u_0) }{ u_0 }
    =
    - g^2 |q|^2 - g^2 |p|^2.
  \end{equation}
  This is in perfect agreement with the corresponding term $\norm{L u_0}^2$
  in the linear case.

  Since the baseline method has an order of accuracy $p$,
  \begin{equation}
    \norm{u_+}^2 - \norm{u_0}^2 = \lot \dt^{p+1} + \O(\dt^{p+2}),
  \end{equation}
  where $\lot$ denotes the leading order term.
  In conclusion, the relaxation parameter $\gamma$ can be expanded as
  \begin{equation}
    \gamma
    =
    1 - \frac{\lot}{g^2 (|q_0|^2 + |p_0|^2)} \dt^{p-1} + \O(\dt^{p}).
  \end{equation}

  To compute the order of accuracy, the expansions
  \cite[eqs.~(311d) and (313c)]{butcher2016numerical}
  \begin{align}
    \label{eq:expansion-u-dt}
    u(\dt)
    &=
    u_0 + \sum_{|t| \leq p+1} \frac{1}{\sigma(t) t!} \dt^{|t|} F(t)(u_0)
    + \O(\dt^{p+2}),
    \\
    \label{eq:expansion-u-gamma-dt}
    u(\gamma \dt)
    &=
    u_0 + \sum_{|t| \leq p+1} \frac{1}{\sigma(t) t!} \gamma^{|t|} \dt^{|t|} F(t)(u_0)
    + \O(\dt^{p+2}),
    \\
    \label{eq:expansion-uplus-gamma}
    u_+^\gamma
    =
    u_0 + \gamma (u_+ - u_0)
    &=
    u_0 + \sum_{|t| \leq p+1} \frac{1}{\sigma(t)} \Phi(t) \gamma \dt^{|t|} F(t)(u_0)
    + \O(\dt^{p+2}),
  \end{align}
  will be used. The order conditions
  \begin{equation}
    \Phi(t) = \frac{1}{t!},
    \qquad
    |t| \leq p,
  \end{equation}
  are satisfied for the baseline RK scheme. Hence,
  \begin{equation}
    u_+^\gamma - u(\gamma \dt)
    =
    (\gamma - \gamma^2) \sum_{|t| = 2} \frac{1}{\sigma(t)} \Phi(t) \dt^2 F(t)(u_0)
    + \sum_{|t| = p+1} \frac{1}{\sigma(t)} \left( \Phi(t) - \frac{1}{t!} \right) \dt^{p+1} F(t)(u_0)
    + \O(\dt^{p+2}).
  \end{equation}
  Inserting the value of $\gamma$ and the only rooted tree $t = \rootedtree[*]$
  of order two,
  \begin{equation}
    (\gamma - \gamma^2) \sum_{|t| = 2} \frac{1}{\sigma(t)} \Phi(t) \dt^2 F(t)(u_0)
    =
    \frac{\dt^{p+1}}{2} \frac{\lot}{g^2 (|q_0|^2 + |p_0|^2)} f'f(u_0).
  \end{equation}
  Using $\|u_0\|^2 = \|u(\dt)\|^2$ and the expansions
  \eqref{eq:expansion-uplus} \& \eqref{eq:expansion-u-dt},
  \begin{multline}
    \lot
    =
    2 \sum_{|t| = p+1} \frac{1}{\sigma(t)} \left( \Phi(t) - \frac{1}{t!} \right) \scp{ F(t)(u_0) }{ u_0 }
    \\
    + \underbrace{\sum_{|t_1| + |t_2| = p+1} \frac{1}{\sigma(t_1) \sigma(t_2)} \left( \Phi(t_1) \Phi(t_2) - \frac{1}{t_1! \, t_2!} \right) \scp{ F(t_1)(u_0) }{ F(t_2)(u_0) }}_{= 0}.
  \end{multline}
  The last sum on the right hand side vanishes, since
  $|t_1|, |t_2| \in \set{1, \dots, p}$
  and consequently $\Phi(t_i) = \nicefrac{1}{t_i!}$ because of the
  order conditions.

  Finally, using Lemma~\ref{lem:nonlinear-Euclidean-Hamiltonian-terms}
  and inserting $f' f(u_0)$, noticing that $\|u_0\|^2 = |q_0|^2 + |p_0|^2$,
  \begin{equation}
  \begin{aligned}
    &\quad
    \frac{1}{2} \frac{\lot}{g^2 (|q_0|^2 + |p_0|^2)} f'f(u_0)
    + \sum_{|t| = p+1} \frac{1}{\sigma(t)} \left( \Phi(t) - \frac{1}{t!} \right) F(t)(u_0)
    \\
    &=
    - \frac{1}{2} \lot \frac{1}{|q_0|^2 + |p_0|^2} u_0
    + \sum_{|t| = p+1} \frac{1}{\sigma(t)} \left( \Phi(t) - \frac{1}{t!} \right) \frac{1}{|q_0|^2 + |p_0|^2} \scp{ F(t)(u_0) }{ u_0 } u_0
    \\
    &=
    0.
  \end{aligned}
  \end{equation}
  Hence, the RRK method has an order of accuracy $p+1$.
\end{proof}

\printbibliography

\end{document}